\documentclass[11pt]{amsart}
\usepackage{amssymb,amsmath}
\usepackage{bm}
\usepackage[]{graphicx}
\usepackage{color}
\usepackage{amssymb, epsfig}
\numberwithin{equation}{section}

\newcommand{\grad}{\nabla }

\numberwithin{figure}{section}
\newtheorem{theorem}{Theorem}[section]

\newtheorem{remark}[theorem]{Remark}
\begin{document}
\title[Existence for the financial stochastic Stefan problem with spread] {Existence of maximal solutions for the financial
stochastic Stefan problem of a volatile asset with spread}

\author[D.~C. Antonopoulou]{D.~C. Antonopoulou$^{\# *}$}
\author[D. Farazakis]{D. Farazakis$^{*}$}
\author[G. Karali]{G. Karali$^{\dag *}$}

\thanks
{$^{\dag}$ Department of Mathematics and Applied Mathematics,
University of Crete, GR--714 09 Heraklion, Greece.}

\thanks
{$^{\#}$ Department of Physical and Mathematical Sciences,
Engineering, University of Chester, CH2 4NU, UK}
\thanks
{$^{*}$ Institute of Applied and Computational Mathematics,
FORTH, GR--711 10 Heraklion, Greece.}

%
%
%

%
%

\begin{abstract}
In this work, we consider
 the outer Stefan problem for the short-time prediction of the spread of
a volatile asset traded in a financial market. The stochastic
equation for the evolution of the density of sell and buy orders
is the Heat Equation with a non-smooth noise in the sense of
Walsh, posed in a moving boundary domain with velocity given by
the Stefan condition. This condition determines the dynamics of
the spread, and the solid phase $[s^-(t),s^+(t)]$ defines the
bid-ask spread area wherein the transactions vanish.
 We introduce a reflection measure and prove existence and uniqueness of maximal
solutions up to stopping times in which the spread
$s^+(t)-s^-(t)$ stays a.s. non-negative and bounded. For this, we
use a Picard approximation scheme and some of the estimates of
\cite{BH} for the Green's function and the associated to the
reflection measure obstacle problem. Analogous results are
obtained for the equation without reflection corresponding to a
signed density. Additionally, we apply some formal asymptotics
when the noise depends only on time to derive that the spread is
given by the integral of the solution of a linear diffusion
stochastic equation.

\end{abstract}
\maketitle \pagestyle{myheadings}
\thispagestyle{plain}

{\small \textbf{Keywords:} Phase field models, Stefan problem,
stochastic volatility, limit order books, spreads.

}
\thispagestyle{plain}
%
%

\section{Introduction}

\subsection{The Stochastic Stefan problem with spread} Let $w(x,t)$ be the
density of sell and buy orders of a stochastically volatile
liquid asset with spread. The moving boundary of the outer Stefan
problem for $w$, $t\geq 0$, is the union of the curves
$x=s^+(t),\;\; x=s^-(t),$ enclosing the solid phase (or spread
area) $\overline{S(t)}$ defined at a given time $t$ by the
interval $\overline{S(t)}:=[s^-(t),s^+(t)].$ The midpoint
$s(t):=(s^-(t)+s^+(t))/2$ is the so-called mid price, and the
length $s^+(t)-s^{-}(t)$ of $\overline{S(t)}$ is the spread at
time $t$. The asset price $x$ has been transformed through a
logarithmic scale and in general can take negative and positive
values. If $x$ is set in $\overline{S(t)}$, then the asset is not
traded, and thus the density $w(x,t)$ of sell and buy orders is
zero, otherwise the order is performed.

The Stefan problem for $w=w(x,t)$ satisfying the stochastic Heat
equation is written as follows
\begin{equation}\label{stef1}\left\{
\begin{aligned}
&\partial_t w=\alpha\Delta w
+\sigma({\rm dist}(x,\partial S))\dot{W}_s,\;\;\;\;x\in\mathbb{R}-\overline{S(t)}\;\;\mbox{(`liquid' phase)},\;\;t>0,\\
&w=0,\;\;x\in
\overline{S(t)}\;\;\;\;\;\;\;\;\;\;\;\;\;\;\;\;\;\;\;\;\;\;\;\;\mbox{(`solid' phase )},\\
&V:=-\nabla w|_{\partial S}\;\;\;\;\;\;\;\;\;\;\;\;\;\;\;\;\;\;\;\;\;\;\;\;\;\;\mbox{(Stefan condition)},\\
&\partial S(0)=\{s^-(0),s^+(0)\}={\rm given}.
\end{aligned}
\right.
\end{equation}
Here, $\alpha> 0$ stands for the total liquidity index of the
market, estimated by the limit order book of the asset, and
$\sigma\dot{W}_s$ is the stochastic volatility. The noise
diffusion $\sigma$ is a function of
\begin{equation}\label{dis}
{\rm dist}(x,\partial S(t))=\min\{|x-s^+(t)|,|x-s^-(t)|\}
\end{equation}
the distance of the price $x$ from the solid phase boundary
$\partial S=\partial S(t)=\{s^-(t),s^+(t)\}$, and
\begin{equation}\label{noi1}
\dot{W}_s(x,t):=\dot{W}(x-s^+(t),t) \;\; \text{if}\;\;x\geq
s^+(t), \;\;
  \dot{W}_s(x,t):=\dot{W}(-x+s^-(t),t) \;\; \text{if}\;\;x\leq s^-(t),
  \end{equation}
where $\dot{W}(\pm x\mp s^\pm(t),t)$ is the non smooth in space
and in time noise defined by Walsh in \cite{Walsh}. The initial
condition $w(x,0)$ is considered given for all $x\in\mathbb{R}$.

The limit orders are instructions for trading of a portion of an
asset, \cite{RobertMcDonh786ald}, based on information from the
limit order book. The lowest sell order $s^+(t)$ defined as ask
price, is the minimum price at which the investor is willing to
receive, and $s^-(t)$ is the highest buy order or bid price which
is the maximum price at which the investor is willing to pay. An
order is executed if the price set (the so-called spot price)
lies outside the spread interval $[s^-(t),s^+(t)]$, if not it is
sorted in the order book list and not traded, see for example in
\cite{GP,anrewload8f,ps}. We also note that the Gibbs Thomson
condition on the moving boundary $\partial S(t)$ which is present
in dimensions $d\geq 2$, \cite{niet1}, involving the mean
curvature, and the constant value of $w=w_0$ in the solid phase
are both replaced by the condition $w=w_0:=0$ in
$\overline{S(t)}$.

The velocity $V$ of
 $\partial S(t)$ is defined at the boundary points by the
 Stefan condition
 \begin{equation}\label{stefcondgen}
 \begin{split}
V(s^+(t),t):=&\partial_t s^{+}(t)=-(\nabla w)^+(s^+(t),t),\\
V(s^-(t),t):=&\partial_t s^{-}(t)=-(\nabla w)^-(s^-(t),t),
\end{split}
\end{equation}
for  $(\nabla \cdot )^\pm$ denoting the derivative from the right
($x> s^+$) and left ($x<s^-$); the Stefan condition describes the
change of liquidity. Therefore, the spread dynamics are given by
\begin{equation}\label{spdyng}
\partial_t s^+(t)-\partial_t s^-(t)=-(\nabla
w)^+(s^+(t),t)+(\nabla w)^-(s^-(t),t)).
\end{equation}
The gradients are taken along the `outer' normal vector, i.e., the
direction is towards the liquid phase, so here in $d=1$ they
coincide to the right and left derivatives. Models with a.s.
non-negative density $w$, when for example a reflection measure is
introduced to the stochastic heat equation, due to the fact that
$w=0$ at $x=s^{\pm}$ will result in an a.s. decreasing spread.
More specifically $(\nabla w)^+(s^+(t),t)\geq 0$ and $(\nabla
w)^-(s^-(t),t)\leq 0$ and thus by \eqref{spdyng} $\partial_t(
s^+(t)- s^-(t))\leq 0$ for all $t\geq 0$ a.s. In contrast, when a
signed density is considered the spread is not monotone.

Motivated by the analysis of \cite{niet1,niet2,AKY} in higher
dimensions, we define the bounded and time independent space
domain $\Omega=(a,b)$ by
\begin{equation}\label{ome}
\Omega=\Omega_{{\rm Liq}}(t)\cup [s^-(t),s^+(t)],
\end{equation}
for a liquid phase $\Omega_{{\rm Liq}}\subset\Omega$ so that for
any $x\in\Omega_{{\rm Liq}}$
\begin{equation}\label{lam}
0\leq |x-s^-|,|x-s^+|\leq\lambda,
\end{equation}
 for $\lambda=b-a$ a positive
constant relatively very larger than the initial spread
$s^+(0)-s^-(0)$. The density $w(x,t)$ will be observed for $x$ in
$\Omega$. As $\lambda\rightarrow\infty$ the liquid phase becomes
infinite as in \eqref{stef1} and $\Omega$ will correspond to
$\mathbb{R}$. The problem is one-dimensional and the liquid phase
consists of two separate bounded linear segments. This enables the
spliting of the Stefan problem equation in two equations posed
for $x\in\Omega_{\rm Liq}$ on $x\geq s^+$ and on $x\leq s^-$ where
we shall apply the change of variables
\begin{equation}\label{cv1gen}
\begin{split}
y= x-s^+(t) \;\; \text{if }x\geq s^+(t), \;\;
    y=-x+s^-(t) \;\; \text{if }x\leq s^-(t),
  \end{split}
\end{equation}
and thus
\begin{equation}\label{cv2gen}
\begin{split}
y_t=  -\partial_t s^+(t) \;\; \text{if }x\geq s^+(t), \;\;
    \partial_ts^-(t) \;\; \text{if }x\leq s^-(t).
  \end{split}
\end{equation}
As we shall see the equation is transformed due to the Stefan
condition into two independent ones posed each on the fixed space
domain $\mathcal{D}:=(0,\lambda)$ with Dirichlet b.c. The value
$y=0$ occurs when the price $x$ is $s^{\pm}$, while $y=\lambda$
when the spread is zero and $s^+=s^-$ hits the boundary of
$\Omega$. These equations are of the general form
\begin{equation}\label{meq11}
v_t(y,t)=\alpha \Delta v(y,t)\mp\nabla v(0^+,t)\nabla
v(y,t)\pm\sigma(y)\dot{W}(y,t)+\dot{\eta}(y,t),\;\;y\in\mathcal{D},\;\;t\geq
0,
\end{equation}
 for $\eta$ a
reflection measure keeping $v$ a.s. non-negative, while $\eta=0$
will correspond to the unreflected problem and a signed $v$. We
also note that when a system is considered in place of the Stefan
problem \eqref{stef1} with buy and sell densities observed
separately and with different liquidity coefficients $\alpha_1$,
$\alpha_2$, the same equation of the above general form will
appear after the change of variables for
$\alpha=\alpha_1,\alpha_2$.

We prove existence of unique maximal solutions $(v,\eta)$ for the
stochastic equation \eqref{meq11} for the stopping time
$\sup_{M>0}\tau_M$ where
\begin{equation}\label{rst}
\begin{split}
\tau_M:=&\inf\Big{\{}T\geq
0:\;\displaystyle{\sup_{r\in(0,T)}}|\nabla v(0^+,r)|\geq
M\Big{\}},
\end{split}
\end{equation}
up to which $|\nabla v(0^+,r)|=\nabla v(0^+,r)$ stays a.s.
bounded. In the case of the unreflected problem, $\eta$ is just
replaced by zero and the absolute value is kept. In order to
return to the initial variables and to the moving boundary
problem, the stopping time will be further reduced so that the
spread stays a.s. non-negative and the spread area in the domain
$\Omega$. These restrictions will be induced by the Stefan
condition and the resulting spread dynamics \eqref{spdyng} on
$w$, the initial spread $s^+(0)-s^-(0)$, and the magnitude of
$\lambda$.

Deterministic parabolic Stefan problems have been so far
extensively studied when describing the phenomenon of phase
separation of alloys. In \cite{niet1}, Niethammer introduced
the deterministic version of \eqref{stef1} in higher dimensions
 in the physical context of the LSW theory for the Ostwald
ripening of alloys; there, a first order approximation was
established for the dynamics of the radii of spherical moving
boundaries in dimensions $d=3$. In \cite{af, afk1, afk2}, the
authors considered the quasi-static problem and obtained second
order approximations by taking into account the variable in
general geometry of the solid phase. We also refer to \cite{AKY}
for the analysis of the parabolic Stefan problem of \cite{niet1}
in the presence of kinetic undercooling and additive forcing.

 Antonopoulou, Bitsaki, and Karali, in \cite{ABiK}, derived the rigorous financial
interpretation of the parabolic Stefan stochastic model, which
applies for a portfolio of assets when $d\geq 2$; a quasi-static
version thereof approximates the parabolic one when the diffusion
tends to infinity as in the case of very large trading. In
contrast to the deterministic Stefan problem where a spherical
initial solid phase or the interval $[s^-(0),s^+(0)]$ in dimension
$d=1$ are static solutions, in the stochastic case the boundary
changes as time evolves due to the random perturbation in the
spde; see for example the numerical simulations in \cite{ABiK}
when $d=3$.
When the sell and buy orders densities 
 are observed separately, then the evolution is described by a system
of two stochastic Heat equations with different liquidity
coefficients and volatilities depending on the distances
$|x-s^+(t)|$, $|x-s^-(t)|$ respectively.
Hambly, and Kalsi proved in \cite{BH} existence and uniqueness of
stochastic solutions for such two phases Stefan systems with
reflection,
 but under the assumption of zero spread for the
asset price, i.e., for $s^+(t)=s^-(t)=s(t)$. Considering 2-phases
1-dimensional stochastic Stefan systems for the evolution of sell
and buy orders without spread we refer also to \cite{Ek,zz}, and
to the more recent results of \cite{mu,BH2}.

\subsection{Limit order books and spread}
An asset is defined as volatile if the corresponding trading
price of sell or buy orders deviates from the mid (mean) price.
The spread's length which is given as the difference between the
actual sell price and the buy price is a measure for the risk of
investment, \cite{AM}. In particular, highly traded assets tend
to have very small spreads, while a relatively large spread
indicates a higher risk. An order is a commitment from the
traders, a buyer or a seller, to buy or sell respectively at an
appropriate price at a given time $t>0$, for which the profit of
the trade is maximized for both sides, \cite{GP}, also called
limit price. The spread and the density of transactions reflect
asset's liquidity. The total volume of active limit orders in a
financial market at a given time is stored in the asset's limit
order book. The liquidity coefficient $\alpha>0$ in the
Stochastic Heat equation of the Stefan problem measures the
diffusion strength of sell and buy orders and can be
approximated, in small time periods, by the total volume of
orders divided by an average spread, \cite{ABiK}.

The various types of financial contracts require
 two entities, that is the holder of a financial
asset, such as a security, commodity, or currency, who receives
the future payments, and the issuer side which has the obligation
to deliver the payments according to the initial terms and claims
of transaction. An order is a commitment to buy or sell at a
given time. Market orders are executed immediately upon
submission in contrast to limit orders which remain active until
they achieve the expected `closing' price; this is an automatic
execution procedure via online electronic platforms. Electronic
trading platforms offer the ability to trade upon information
from historical data like past market prices and curves of prices
of stocks (old trading view). The market orders are based on the
current market prices, while limit orders target to better future
prices for maximizing profits, \cite{RobertMcDonh786ald}. Limit
orders are low risk commitments since the price of execution for
sell or buy is predetermined
 reducing thus the odds of significant failure. However,
the process is time consuming and the order may never be
 executed.

In \cite{connectliq} the German power market liquidity was
studied, we also refer to \cite{liqspread} for a statistical
analysis of the fluctuations of the average spread where the
relation of spread with shares volume and volatility was
examined, or to \cite{evolutbidask} for a stochastic equation
model estimating the liquidity risk. In \cite{sc88fee}, the
authors analyzed how transaction costs affect the spreads while in
case of zero cost then the market price should act as a Wiener
process; see also in \cite{trnacost} for the liquidity risk with
respect to the transaction costs and market manipulation under a
Brownian motion problem formulation, or in
\cite{mfkorott,Rollmeth, Hasbrouckt,Bleaney}, and in
\cite{Ekinci} for various empirical approaches on spread's
forecast. We note that except from the bid-ask spread, there
exist several other types of spread like the asset swap spread,
the yield spread, the zero volatility spread, the option adjusted
spread, the default swap spread, or the bank spreads, see for
example in \cite{domocane,Berd1,Berd2,Berd3,gropbank,Hosaunders}.

\subsection{Main Results}
Our analysis covers 3 versions of the Stefan problem.
\begin{enumerate}
\item
Let $w_1,w_2\geq 0$ be the density of sell orders and buy orders
respectively. When $x> s^+(t)$ then only sell orders are executed
($w_2=0$), while when $x< s^-(t)$ then only buy orders are
executed ($w_1=0$). Moreover at $x=s^+(t)$ $w_1=0$ and at
$x=s^-(t)$ $w_2=0$. The signed density $w=w_1-w_2$ is given by
\begin{equation}\label{sd1}
w(x,t)= w_1(x,t) \;\;\text{if}\;\;x> s^+(t), \;\;w(x,t)=
    -w_2(x,t) \;\;\text{if}\;\; x< s^-(t),\;\;
    w(x,t)=0\;\; \text{otherwise}.
\end{equation}
We introduce in \eqref{stef1} the additive term $\dot{\eta}_s$
defined by
\begin{equation}\label{m1}
\dot{\eta}_s(x,t):=
      \dot{\eta}_1(x-s^+(t),t)\;\; \text{if}\;\;x\geq s^+(t),\;\;
    \dot{\eta}_s(x,t):=-\dot{\eta}_2(-x+s^-(t),t)\;\; \text{if}\;\;x\leq s^-(t),
\end{equation}
    where $\eta_1$,
$\eta_2$ are reflection measures so that $w_1,w_2\geq 0$.
\item
We consider the reflected problem where $w\geq 0$. The Stefan
condition due to the non-negativity of $w$ which vanishes at
$x=s^\pm$ yields an a.s. decreasing spread. The reflection
additive term on \eqref{stef1} is of the form
\begin{equation}\label{m11} \dot{\eta}_s(x,t):=
    \dot{\eta}_1(x-s^+(t),t)\;\; \text{if}\;\;x\geq s^+(t), \;\;
    \dot{\eta}_s(x,t):=\dot{\eta}_2(-x+s^-(t),t) \;\; \text{if}\;\;x\leq s^-(t),
\end{equation}
    where $\eta_1$,
$\eta_2$ are reflection measures keeping $w\geq 0$ for any
$x\in\Omega_{\rm Liq}$.
\item The unreflected
problem is analyzed with a signed density $w$ where as in (1) the
spread is non-monotone.
\end{enumerate}

In all the above cases we derive a system of independent spdes of
the form \eqref{meq11} for $v=v_1$, $v=v_2$. Then $s^+,s^-$ are
specified through integration of the Stefan condition. For
stopping times wherein $s^-\leq s^+$ and
$(s^-,s^+)\subseteq\Omega$ by applying the change of variables
\eqref{cv1gen}, $v_1\rightarrow w|_{x\geq s^+}$, $v_2\rightarrow
-w|_{x\leq s^-}$ in (1) or $w|_{x\leq s^-}$ in (2) and (3), we
return to the initial Stefan problem. The suggested
transformation is efficient on representing the stochastic
equation of the Stefan problem as a system of independent spdes
posed on the fix domain $\mathcal{D}=(0,\lambda)$, of the same
general form. Additionally, for the reflected equations, we
impose the non-negativity of $v_{1,2}$ by proving existence of
the measures $\eta_{1,2}$ on the fix domain which then define the
additive reflection term in the initial equation. Our novel
approach on transforming first the problem to an spde of
reference and then establishing maximal solutions to the initial
one by using the Stefan condition for the stopping times is also
applicable for various other one-dimensional versions with
financial interest being analyzed for example in \cite{BH,BH2,mu}
without spread. Note that our model permits zero spread. The
noise diffusion and the noise depend on the distance of $x$ from
the spread area boundary and not on the position of $x$. This
yields, since the velocity is given by the standard Stefan
condition of \eqref{stef1}, to spdes in the $y$ variables where
$s^{\pm}$, that belong to the initial problem unknowns, are
absent. We also mention that variables of the form $y=-x+s(t)$
when $x\leq s$, $y=x-s(t)$ when $x\geq s$ for the zero spread
model where $s$ is the sell/buy price, are used in \cite{BH,mu}.
In \cite{BH} the problem is not transformed, a weak solution
formulation for proper test functions compactly supported in
$[0,1]$ induces somehow a relevant (not splited) system posed on
the fixed domain $(0,1)$ that seems to facilitate the authors
proof of maximal solutions.

In Section 2, we present analytically the change of variables
$y=x-s^+$, $y=-x+s^-$ for $x\in\Omega_{\rm Liq}$, use the Stefan
condition and derive per case the Stefan problems as systems of
two independent spdes of the form \eqref{meq11} for $v=v_1,v_2$,
cf. \eqref{stefr1y}, \eqref{stefr2y}, \eqref{stefr3y}. Section 3
is devoted to the existence of unique weak maximal solutions
$(v,\eta)$ of the Dirichlet problem on $\mathcal{D}$ for
\eqref{meq11} with reflection, and then of maximal solutions to
the initial variables with stopping times restricted by the
Stefan condition, the non-negativity of spread and the
boundedness of the liquid phase. In detail, we write the spde in
an integral form using the Green's function of the negative
Dirichlet Laplacian and construct an approximate Picard scheme
for the truncated problem. In Theorem \ref{t1} we prove existence
and uniqueness a.s. for the Picard approximations, and on the
limit existence and uniqueness of the truncated solution. For
this, we use some of the Green's estimates of \cite{BH} and a
proper Banach space introduced therein. The reflection measure
$\eta$ is associated to the obstacle problem estimated in
\cite{BH}. In Theorem \ref{t2}, using the consistency of the
truncated solutions we prove existence of a unique maximal
solution $(v,\eta)$ a.s. in the maximal time interval
$[0,\displaystyle{\sup_{M>0}}\tau_M)$ for $\tau_M$ given by
\eqref{rst}. Given the maximal solution $(v,\eta)$, for
$v=v_{1,2}\geq 0$, $\eta=\eta_{1,2}$, in Theorem \ref{t31} we
prove existence of unique maximal solutions $(w_1,\eta_1)$,
$(w_2,\eta_2)$ to the reflected Stefan problem
\eqref{stefr}-\eqref{psi}-\eqref{refl} corresponding to (1), and
of $w|_{x\geq s^+}=w_1\geq 0$, $w|_{x\leq s^-}=-w_2\leq 0$, in the
maximal interval $\mathcal{I}_1:=[0,\hat{\tau})$ for $\hat{\tau}:=
\min\{\displaystyle{\sup_{M>0}}\tau_{1M},\tau_{1s},\tau_1^*\}$,
with $\tau_{1M},\tau_{1s},\tau_1^*$ given by \eqref{rtauM1},
\eqref{psp1}, \eqref{nha1} for which the spread exists and stays
a.s. non-negative for any $t\in\mathcal{I}_1$. An analogous result
for the case (2) is proven in Theorem \ref{t32} but in a different
maximal interval $\mathcal{I}_2:=[0,\hat{\tau})$ for $\hat{\tau}:=
\min\{\displaystyle{\sup_{M>0}}\tau_{1M},\tau_{2s}\}$, with
$\tau_{1M},\tau_{2s}$ given by \eqref{rtauM1}, \eqref{psp2}.
There, the decreasing property of the spread is used.

In Section 4 we consider the Stefan problem without reflection,
i.e., (3), and the Dirichlet problem on $\mathcal{D}$ for the spde
\eqref{meqmain3} that $v$ satisfies. Theorem \ref{t1d}
establishes existence and uniqueness a.s. of the truncated
equation, and Theorem \ref{t2d} existence of a unique maximal
solution $v$ in $[0,\displaystyle{\sup_{M>0}}\tau_M)$ for
$\tau_M$ as in \eqref{rst}. Then the existence and uniqueness of
maximal solution in the initial variables is proven in Theorem
\ref{t33} for the resulting stopping time. We also present some
formal asymptotics for very large liquidity coefficient $\alpha$,
when the noise is only time dependent (for example the formal
derivative of a Brownian) and for constant noise diffusion. Under
the assumption that $w$ approximates a mean-field value
$w_\infty(t)$ when the distance from the spread boundary is very
large, we derive that the spread is given by the integral of the
solution of a stochastic linear diffusion equation, see the
dynamics in \eqref{form1}, and the linear sde \eqref{form2}.

\section{The Stefan problems}
\subsection{Change of variables}
We consider $\Omega$ given by \eqref{ome}, $\Omega_{\rm Liq}$ by
\eqref{lam}, and $y$ defined by \eqref{cv1gen} for any
$x\in\Omega_{\rm Liq}\cup\{s^-,s^+\}$. Let $\tilde{w}_1(x,t)$ be
defined in $\{x\in\Omega_{\rm Liq}\cup\{s^-,s^+\}:\;x\geq s^+\}$
and $\tilde{w}_2(x,t)$ be defined in $\{x\in\Omega_{\rm
Liq}\cup\{s^-,s^+\}:\;x\leq s^-\}$ and set for $y:=x-s^+$
$$\tilde{w}_1(x,t):=\tilde{v}_1(y,t)\;\; \forall\;x\in\Omega_{\rm
Liq}\cup\{s^-(t),s^+(t)\}:\;x\geq s^+(t),$$ while for $y:=-x+s^-$
$$\tilde{w}_2(x,t):=\tilde{v}_2(y,t)\;\; \forall\;x\in\Omega_{\rm
Liq}\cup\{s^-(t),s^+(t)\}:\;x\leq s^-(t).$$ If $x\geq s^+(t)$ we
get
\begin{equation}\label{eqs1}
\begin{split}
&\tilde{w}_1(x,t)=\tilde{v}_1(x-s^+(t),t)=\tilde{v}_1(y,t),\;\;\;y=x-s^+(t),\;\;y_x=1,\;\;y_t=-\partial_ts^+(t)\\
&(\tilde{w}_1)_t(x,t)=(\tilde{v}_1)_y(y,t)y_t(y,t)+(\tilde{v}_1)_t(y,t)
=-\partial_ts^+(t)
(\tilde{v}_1)_y(y,t)+(\tilde{v}_1)_t(y,t),\\
&(\tilde{w}_1)_x(x,t)=(\tilde{v}_1)_y(y,t)y_x=+(\tilde{v}_1)_y(y,t),\\
&(\tilde{w}_1)_{xx}(x,t)=(\tilde{v}_1)_{yy}(y,t)(y_x(y,t))^2=(\tilde{v}_1)_{yy}(y,t),
\end{split}
\end{equation}
and if $x\leq s^-(t)$
\begin{equation}\label{eqs2}
\begin{split}
&\tilde{w}_2(x,t)=\tilde{v}_2(-x+s^-(t),t)=\tilde{v}_2(y,t),\;\;y=-x+s^-(t),\;\;y_x=-1,\;\;y_t=\partial_ts^-(t)\\
&(\tilde{w}_2)_t(x,t)=(\tilde{v}_2)_y(y,t)y_t(y,t)+(\tilde{v}_2)_t(y,t)
=\partial_ts^-(t)
(\tilde{v}_2)_y(y,t)+(\tilde{v}_2)_t(y,t),\\
&(\tilde{w}_2)_x(x,t)=(\tilde{v}_2)_y(y,t)y_x=-(\tilde{v}_2)_y(y,t),\\
&(\tilde{w}_2)_{xx}(x,t)=(\tilde{v}_2)_{yy}(y,t)(y_x(y,t))^2=(\tilde{v}_2)_{yy}(y,t).
\end{split}
\end{equation}

\subsection{Case 1}

Let for any $x\in\Omega$ the signed density $w$ be given by
$$w(x,t)=w_1(x,t)-w_2(x,t)=
  \begin{cases}
    w_1(x,t) & \text{if} \;\;x> s^+(t), \\
    -w_2(x,t) & \text{if}\;\; x< s^-(t),\\
    0        & \text{otherwise},
  \end{cases}$$
for $w_1,w_2$ the densities of sell orders and buy orders
respectively. We then have $w(x,t)|_{x\geq s^+}=w_1(x,t)$,
$w(x,t)|_{x\leq s^-}=-w_2(x,t)$.

The equation \eqref{stef1} by introducing the additive term
$\dot{\eta}_s$ given by \eqref{m1} takes in $\Omega_{\rm{Liq}}$
the form
$$\partial_t w=\alpha\Delta w +\sigma({\rm dist}(x,\partial
S))\dot{W}_s(x,t)+\dot{\eta}_s(x,t),\;\;\;\;\;\;\;\;x\in\Omega_{\rm{Liq}},\;\;t>0,$$
or equivalently for $x\in\Omega_{\rm Liq}$
\begin{equation}\label{stefr}
\begin{split}
\partial_t
w_1=&\alpha\Delta w_1
+\sigma(x-s^+(t))\dot{W}(x-s^+(t),t)+\dot{\eta}_1(x-s^+(t),t),\;\;\;\;\;\;\;\;\;\;\;x>
s^+(t),\;\;t>0,\\
\partial_t
w_2=&\alpha\Delta w_2
-\sigma(-x+s^-(t))\dot{W}(-x+s^-(t),t)+\dot{\eta}_2(-x+s^-(t),t),\;\;x<
s^-(t),\;\;t>0,
\end{split}
\end{equation}
while
$w(x,t)=w_1(x,t)=w_2(x,t)=0,\;\;\;\forall\;x\in[s^-(t),s^+(t)],\;\;\forall\;t>0,$
and $w_1(x,0)=w(x,0)$ for any $x\geq s^+(0)$, $w_2(x,0)=-w(x,0)$
for any $x\leq s^-(0)$. We shall assume that
$w_1(x,0),w_2(x,0)\geq 0$. The reflection measures $\eta_1$,
$\eta_2$ if exist will keep $w_1,w_2\geq 0$ for all $t$ a.s.
Using the Stefan condition \eqref{stefcondgen}, we obtain
\begin{equation}\label{stefcond1}
 \begin{split}
V(s^+(t),t)=&\partial_t s^{+}(t)=-(\nabla w)^+(s^+(t),t)=-(\nabla
w_1)^+(s^+(t),t),\\
V(s^-(t),t)=&\partial_t s^{-}(t)=-(\nabla w)^-(s^-(t),t)=(\nabla
w_2)^-(s^-(t),t),
\end{split}
\end{equation}
and so the spread dynamics are given by
\begin{equation}\label{spdyn1}
\partial_t s^+(t)-\partial_t s^-(t)=-(\nabla
w_1)^+(s^+(t),t)-(\nabla w_2)^-(s^-(t),t)).
\end{equation}

We apply the change of variables $w_1(x,t)=v_1(y,t)$ for
$y=x-s^+$ and so $(\nabla w_1)^+(s^+,t)=\nabla v_1(0^+,t)$, and
$w_2(x,t)=v_2(y,t)$ for $y=-x+s^-$ and so $(\nabla
w_2)^-(s^-,t)=-\nabla v_2(0^+,t)$, use \eqref{eqs1}, \eqref{eqs2},
and \eqref{stefcond1} which yields that $\partial_ts^+(t)=-(\nabla
w_1)^+(s^+(t),t)=-\nabla v_1(0^+,t)$, and that
$\partial_ts^-(t)=(\nabla w_2)^-(s^-(t),t)=-\nabla v_2(0^+,t)$,
and derive the system of two independent initial and boundary
value problems
\begin{equation}\label{stefr1y}
\begin{split}
\partial_t
v_1(y,t)=&\alpha\Delta v_1(y,t)+\partial_ts^+(t)\nabla v_1(y,t)
+\sigma(y)\dot{W}(y,t)+\dot{\eta}_1(y,t)\\
=&\alpha\Delta v_1(y,t)-\nabla v_1(0^+,t)\nabla v_1(y,t)
+\sigma(y)\dot{W}(y,t)+\dot{\eta}_1(y,t),\;\;y\in(0,\lambda)=:\mathcal{D},\;\;t>0,\\
&v_1(0,t)=v_1(\lambda,t)=0,\;\;t>0,\;\;v_1(y,0)=w(y+s^+(0),0)\geq 0,\;\;y\in\mathcal{D},\\
&{\rm and}\\
\partial_t
v_2(y,t)=&\alpha\Delta v_2(y,t)-\partial_t s^-(t)\nabla v_2(y,t)
-\sigma(y)\dot{W}(y,t)+\dot{\eta}_2(y,t)\\
=&\alpha\Delta v_2(y,t)+\nabla v_2(0^+,t)\nabla v_2(y,t)
-\sigma(y)\dot{W}(y,t)+\dot{\eta}_2(y,t),\;\;y\in(0,\lambda)=:\mathcal{D},\;\;t>0,\\
&v_2(0,t)=v_2(\lambda,t)=0,\;\;t>0,\;\;v_2(y,0)=-w(s^-(0)-y,0)\geq 0,\;\;y\in\mathcal{D},\\
\end{split}
\end{equation}
where we used the Dirichlet b.c.
$v_1(\lambda,t)=v_2(\lambda,t)=0$.

By using \eqref{spdyn1}, the spread evolution is given by
\begin{equation}\label{spdyn11}
\partial_t (s^+(t)-s^-(t))=-\nabla v_1(0^+,t)+\nabla v_2(0^+,t).
\end{equation}

\subsection{Case 2}

Let for any $x\in\Omega$ the density $w$, and define
$$w_1(x,t):=w(x,t)\;\;x\geq s^+,\;\;w_2(x,t)=w(x,t)\;\;x\leq s^-,$$
and so
$$w(x,t)=\begin{cases}
    w_1(x,t) & \text{if}\;\;x> s^+(t), \\
    w_2(x,t) & \text{if}\;\;x< s^-(t),\\
    0        & \text{otherwise}.
  \end{cases}$$
In this case, we have $w(x,t)|_{x\geq s^+}=w_1(x,t)$,
$w(x,t)|_{x\leq s^-}=w_2(x,t)$.

We introduce in \eqref{stef1} the additive term
$\dot{\eta}_s(x,t)$ given by \eqref{m11}, for $\eta_1$, $\eta_2$
there reflection measures keeping $w_1,w_2\geq 0$ and thus $w\geq
0$. The equation on $\Omega_{\rm Liq}$ takes the form
$$\partial_t w=\alpha\Delta w +\sigma({\rm dist}(x,\partial
S))\dot{W}_s(x,t)+\dot{\eta}_s(x,t),\;\;\;\;\;\;\;\;x\in\Omega_{\rm
Liq},\;\;t>0,$$ or equivalently for $x\in\Omega_{\rm Liq}$
\begin{equation}\label{stefrr}
\begin{split}
\partial_t
w_1=&\alpha\Delta w_1
+\sigma(x-s^+(t))\dot{W}(x-s^+(t),t)+\dot{\eta}_1(x-s^+(t),t),\;\;\;\;\;\;\;\;\;\;\;x>
s^+(t),\;\;t>0,\\
\partial_t
w_2=&\alpha\Delta w_2
+\sigma(-x+s^-(t))\dot{W}(-x+s^-(t),t)+\dot{\eta}_2(-x+s^-(t),t),\;\;x<
s^-(t),\;\;t>0,
\end{split}
\end{equation}
while
$w(x,t)=w_1(x,t)=w_2(x,t)=0,\;\;\;\forall\;x\in[s^-(t),s^+(t)],\;\;\forall\;t>0,$
and $w_1(x,0)=w(x,0)$ for any $x\geq s^+(0)$, $w_2(x,0)=w(x,0)$
for any $x\leq s^-(0)$. We shall assume that
$w_1(x,0),w_2(x,0)\geq 0$. The reflection measures $\eta_1$,
$\eta_2$ if exist will keep $w_1,w_2,w\geq 0$ for all $t$ a.s.
Using the Stefan condition \eqref{stefcondgen}, we obtain
\begin{equation}\label{stefcond2}
 \begin{split}
V(s^+(t),t)=&\partial_t s^{+}(t)=-(\nabla w)^+(s^+(t),t)=-(\nabla
w_1)^+(s^+(t),t),\\
V(s^-(t),t)=&\partial_t s^{-}(t)=-(\nabla w)^-(s^-(t),t)=-(\nabla
w_2)^-(s^-(t),t),
\end{split}
\end{equation}
and so the spread dynamics are given by
\begin{equation}\label{spdyn2}
\partial_t s^+(t)-\partial_t s^-(t)=-(\nabla
w_1)^+(s^+(t),t)+(\nabla w_2)^-(s^-(t),t)).
\end{equation}

We apply the change of variables $w_1(x,t)=v_1(y,t)$ for
$y=x-s^+$ and so $(\nabla w_1)^+(s^+,t)=\nabla v_1(0^+,t)$, and
$w_2(x,t)=v_2(y,t)$ for $y=-x+s^-$ and so $(\nabla
w_2)^-(s^-,t)=-\nabla v_2(0^+,t)$, use \eqref{eqs1}, \eqref{eqs2},
and \eqref{stefcond1} which yields that $\partial_ts^+(t)=-(\nabla
w_1)^+(s^+(t),t)=-\nabla v_1(0^+,t)$, and that
$\partial_ts^-(t)=-(\nabla w_2)^-(s^-(t),t)=\nabla v_2(0^+,t)$, to
derive the system of two independent initial and boundary value
problems
\begin{equation}\label{stefr2y}
\begin{split}
\partial_t
v_1(y,t)=&\alpha\Delta v_1(y,t)+\partial_ts^+(t)\nabla v_1(y,t)
+\sigma(y)\dot{W}(y,t)+\dot{\eta}_1(y,t)\\
=&\alpha\Delta v_1(y,t)-\nabla v_1(0^+,t)\nabla v_1(y,t)
+\sigma(y)\dot{W}(y,t)+\dot{\eta}_1(y,t),\;\;y\in(0,\lambda)=:\mathcal{D},\;\;t>0,\\
&v_1(0,t)=v_1(\lambda,t)=0,\;\;t>0,\;\;v_1(y,0)=w(y+s^+(0),0)\geq 0,\;\;y\in\mathcal{D},\\
&{\rm and}\\
\partial_t
v_2(y,t)=&\alpha\Delta v_2(y,t)-\partial_t s^-(t)\nabla v_2(y,t)
+\sigma(y)\dot{W}(y,t)+\dot{\eta}_2(y,t)\\
=&\alpha\Delta v_2(y,t)-\nabla v_2(0^+,t)\nabla v_2(y,t)
+\sigma(y)\dot{W}(y,t)+\dot{\eta}_2(y,t),\;\;y\in(0,\lambda)=:\mathcal{D},\;\;t>0,\\
&v_2(0,t)=v_2(\lambda,t)=0,\;\;t>0,\;\;v_2(y,0)=w(s^-(0)-y,0)\geq 0,\;\;y\in\mathcal{D}.\\
\end{split}
\end{equation}

By using \eqref{spdyn1}, the spread evolution is given by
\begin{equation}\label{spdyn22}
\partial_t (s^+(t)-s^-(t))=-\nabla v_1(0^+,t)-\nabla v_2(0^+,t).
\end{equation}
As we already mentioned, if the reflection measures exist and keep
$v_1,v_2\geq 0$ and since $v_1=v_2=0$ at $y=0$, then the spread is
decreasing.
\subsection{Case 3}
As in Case 2, we consider for any $x\in\Omega$ the signed density
$w$, and define
$$w_1(x,t):=w(x,t)\;\;x\geq s^+,\;\;w_2(x,t)=w(x,t)\;\;x\leq
s^-.$$ We do not require here $w\geq 0$ and so we consider the
unreflected equation \eqref{stef1} posed in $\Omega_{\rm
Liq},\;\;t>0,$ or equivalently for $x\in\Omega_{\rm Liq}$
\begin{equation}\label{stefrrr}
\begin{split}
\partial_t
w_1=&\alpha\Delta w_1
+\sigma(x-s^+(t))\dot{W}(x-s^+(t),t),\;\;\;\;\;\;\;\;\;\;\;x>
s^+(t),\;\;t>0,\\
\partial_t
w_2=&\alpha\Delta w_2
+\sigma(-x+s^-(t))\dot{W}(-x+s^-(t),t),\;\;x< s^-(t),\;\;t>0,
\end{split}
\end{equation}
while
$w(x,t)=w_1(x,t)=w_2(x,t)=0,\;\;\;\forall\;x\in[s^-(t),s^+(t)],\;\;\forall\;t>0,$
and $w_1(x,0)=w(x,0)$ for any $x\geq s^+(0)$, $w_2(x,0)=w(x,0)$
for any $x\leq s^-(0)$. Using the Stefan condition
\eqref{stefcondgen}, we obtain \eqref{stefcond2} again for the
velocity and the spread dynamics are given by \eqref{spdyn2}. We
apply the change of variables $w_1(x,t)=v_1(y,t)$ for $y=x-s^+$,
and $w_2(x,t)=v_2(y,t)$ for $y=-x+s^-$ to obtain as in Case 2 the
system of two independent initial and boundary value problems
\begin{equation}\label{stefr3y}
\begin{split}
\partial_t
v_1(y,t)=&\alpha\Delta v_1(y,t)-\nabla v_1(0^+,t)\nabla v_1(y,t)
+\sigma(y)\dot{W}(y,t),\;\;y\in(0,\lambda)=:\mathcal{D},\;\;t>0,\\
&v_1(0,t)=v_1(\lambda,t)=0,\;\;t>0,\;\;v_1(y,0)=w(y+s^+(0),0),\;\;y\in\mathcal{D},\\
&{\rm and}\\
\partial_t
v_2(y,t)=&\alpha\Delta v_2(y,t)-\nabla v_2(0^+,t)\nabla v_2(y,t)
+\sigma(y)\dot{W}(y,t),\;\;y\in(0,\lambda)=:\mathcal{D},\;\;t>0,\\
&v_2(0,t)=v_2(\lambda,t)=0,\;\;t>0,\;\;v_2(y,0)=w(s^-(0)-y,0),\;\;y\in\mathcal{D}.\\
\end{split}
\end{equation}
The spread evolution is given as in Case 2 by \eqref{spdyn22}, but
since $v_1$, $v_2$ may change sign even if $v_1=v_2=0$ at $y=0$,
the spread is not monotone.

\vspace{0.5cm}

When reflection measures are considered, i.e., for the Cases 1,2,
each problem's unknowns for $t\in[0,T]$ is a pair $(v,\eta)$ where
the reflection measure $\eta$ is defined to satisfy
\begin{equation}\label{psi}
\begin{split}
&\mbox{for all}\mbox{ measurable functions }
\psi:\;\overline{\mathcal{D}}\times (0,T)\rightarrow [0,\infty)\\
&\int_0^t\int_{\mathcal{D}}\psi(y,s)\eta(dy,ds)\;\;\mbox{is}\;\;\mathcal{F}_t-\mbox{measurable},
\end{split}
\end{equation}
and the constraint
\begin{equation}\label{refl}
\int_0^T\int_{\mathcal{D}} v(y,s)\eta(dy,ds)=0.
\end{equation}

We shall assume that the noise diffusion $\sigma$ is a sufficiently smooth 
function; its minimum regularity will be specified in the sequel. 
The random measure $W(dy,ds)$ is defined as the $1$-dimensional
space-time white noise induced by the $2$-dimensional Wiener
process $W:=\{W(y,t):\;t\in[0,T],\;y\in (0,\lambda)\}$ which
generates, for any $t\geq 0$, the filtration
$\mathcal{F}_t:=\sigma(W(y,s):\;s\leq t,\;y\in(0,\lambda))$,
where the notation $\sigma$ here denotes the $\sigma$-algebra.

\begin{remark}
In all the above cases, given the solutions $v_1,v_2$ for
$t\in[0,T]$, $s^+(t),\;s^-(t)$ and the spread $s^+(t)-s^-(t)$ are
derived by direct formulae after integration of the Stefan
condition in $[0,t]$.
\end{remark}
\begin{remark}
We observe that the transformed spdes of Cases 1,2,3 are of the
general form \eqref{meq}, i.e.,
\begin{equation*}\label{meq}
v_t(y,t)=\alpha \Delta v(y,t)\mp\nabla v(0^+,t)\nabla
v(y,t)\pm\sigma(y)\dot{W}(y,t)+\dot{\eta}(y,t),
\end{equation*}
posed on $\mathcal{D}:=(0,\lambda)$ for $t\in[0,T]$, with
Dirichlet b.c. $v(x,t)=0$ at $\partial\mathcal{D}$, and $v(y,0)$
given, for $v\geq 0$ when $\eta$ not the zero measure, and signed
$v$ when $\eta\equiv 0$.
\end{remark}
\begin{remark}
Given $v_{1,2}$ for any $y\in\mathcal{D}$ and any $t$ in $[0,T]$,
then the Stefan condition will determine after integration
$s^\pm(t)$ in $(0,T]$. Let $x\in\Omega_{\rm Liq}$ then for any
given $t\in[0,T]$ and any $x\geq s^+(t)$ since $y=x-s^+(t)$,
$w(x,t)$ will be defined by $v_1(x-s^+(t),t)$, while for any
$x\leq s^-(t)$ since $y=-x+s^-(t)$, $w(x,t)$ will be defined by
$-v_2(-x+s^-(t),t)$ for case (1) or by $v_2(-x+s^-(t),t)$ for
Cases 2,3.
\end{remark}
\begin{remark}
Evolution for $v_{1,2}$ will be observed as long as $a\leq s^-\leq
s^+\leq b$, while $a\leq x\leq b$. In particular, consider
$y=\lambda=b-a$. Then if $x\geq s^+$ then $y=b-a=x-s^+\leq b-s^+$
will yield $-a\leq -s^+$ i.e., $s^+\leq a$ and thus $s^+=a$ and
$s^-=s^+=a$ and $x=b$ which is the case when the the spread is
zero and hits the boundary at $b$ and $v_1(\lambda,t)=0$. If
$x\leq s^-$ then $y=b-a=-x+s^-\leq -a+s^-$ will yield $b\leq s^-$
and thus $s^-=b$ and $s^+=s^-=b$ and $x=a$ which is the case when
the the spread is zero and hits the boundary at $a$ and
$v_2(\lambda,t)=0$. When $y=0$ then either $x=s^-$ and
$v_2(0,t)=0$ or $x=s^+$ and $v_1(0,t)=0$. For all $x\in(s^-,s^+)$
the density $w(x,t)$ will be set to $0$. The initial values of
$v_{1,2}$ are well defined through the initial value $w(x,0)$
which is given for all $x\in\mathbb{R}$. We assume that $w(x,0)$
is compactly supported in $\overline{\Omega}$ to obtain a
compatibility condition to $v_{1,2}(\lambda,t)=0$ at $t=0$.

We will analyze in detail in the sequel how the restrictions of a
non-negative spread and spread area in the domain $\Omega$, i.e.,
$a< s^-\leq s^+< b$, reduce the stopping time up to which maximal
solutions $w_{1,2}$ exist.

\end{remark}

\section{Existence of maximal solutions with reflection}
In what follows we shall present the analytical proof of
existence of unique maximal solutions $(v,\eta)$ for the initial
and boundary value problem for
\begin{equation}\label{meqmain}
v_t(y,t)=\alpha \Delta v(y,t)-\nabla v(0^+,t)\nabla
v(y,t)+\sigma(y)\dot{W}(y,t)+\dot{\eta}(y,t),
\end{equation}
posed for any $y$ in $\mathcal{D}=(0,\lambda)$ for $t\in[0,T]$
with Dirichlet b.c., with $v(y,0)$ given, and $\eta$ a reflection
measure satisfying \eqref{psi} and \eqref{refl} keeping $v$
non-negative. As $\alpha>0$ the proof for the 2d i.b.v. problem
of \eqref{stefr1y} of Case 1 is completely analogous, while the
results for Case 3 (unreflected problem) will be derived at a next
section by setting $\eta\equiv 0$. We will keep the absolute
values on $\nabla v(0^+,t)$ appearing in the following proofs
(even if non-negative in \eqref{meqmain}) so that the results are
applicable for these cases directly.
\subsection{Weak formulation} Let us define an $L^2(\mathcal{D})$
basis of eigenfunctions
$w_n:=\sin\Big{(}\frac{n\pi}{\lambda}x\Big{)},\;\;n=0,1,2,\cdots,$
corresponding to the eigenvalues $\mu_n$, $n=0,1,\cdots$ of
$-\Delta u=\mu u,\;\;u(0)=0,\;\;u(\lambda)=0,$ where
$\mu_n:=\frac{n^2\pi^2}{\lambda^2},\;\;n=0,1,2,\cdots$ The
associate Green's function for the negative of the Dirichlet
Laplacian can then be given by
$\frac{2}{\lambda}\displaystyle{\sum_{n=0}^{\infty}}e^{-\mu_n
t}w_n(x)w_n(y),$ see \cite{CH1}, so that the Green's function
corresponding to $-\alpha\Delta$ with Dirichlet b.c. is given by
$G(t,x,y)=\frac{2}{\lambda}\displaystyle{\sum_{n=0}^{\infty}}e^{-\alpha\mu_n
t}w_n(x)w_n(y)$.

We say that $v$ is a weak (analytic) solution of \eqref{meqmain}
if it satisfies for all $\phi=\phi(y)$ in
$C^2(\overline{\mathcal{D}})$ with $\phi(0)=\phi(\lambda)=0$, the
following weak formulation
\begin{align}\label{d1}
\int_{\mathcal{D}}\Big(v(y,t)-v_0(y)\Big)\phi(y)dy=&
\int_0^t \int_{\mathcal{D}}\Big(
\alpha\Delta \phi(y)v(y,s)
+\nabla\phi(y)\nabla v(0^+,s) v(y,s)\Big{)}dyds \nonumber \\
&+\int_0^t\int_{\mathcal{D}}\phi(y)\sigma(y)W(dy,
ds)+\int_0^t\int_{\mathcal{D}}\phi(y)\eta(dy, ds).
\end{align}

The solution of \eqref{meqmain} admits for any $y\in\mathcal{D}$,
$t\in[0,T]$, the next integral representation
\begin{equation}\label{g1}
\begin{split}
v(y,t)=&\int_{\mathcal{D}}v_0(z)G(y,z,t)dz\\
&+\int_0^t\int_{\mathcal{D}}\nabla v(0^+,s)\nabla G(y,z,t-s) v(z,s)dzds\\
&+\int_0^t\int_{\mathcal{D}} G(y,z,t-s)\sigma(z)W(dz,
ds)+\int_0^t\int_{\mathcal{D}} G(y,z,t-s)\eta(dz, ds),
\end{split}
\end{equation}
and $\eta$ satisfies \eqref{psi}, \eqref{refl}.
\subsection{Main Theorems}
Let the Banach space $(\mathcal{B},\|\cdot\|_{\mathcal{B}})$
$$\mathcal{B}:=\Big{\{}f\in C(\overline{\mathcal{D}}):\;\exists
f'(0),\;f(0)=f(\lambda)=0\Big{\}},$$ with the norm
$\|\cdot\|_{\mathcal{B}}:\mathcal{B}\rightarrow\mathbb{R}^+$,
defined by
$$\|f\|_{\mathcal{B}}:=\displaystyle{\sup_{y\in\mathcal{D}}}\Big{|}\frac{f(y)}{y}\Big{|}.$$

Let $M>0$ fixed, we define in the Banach space $\mathcal{B}$, as
in \cite{BH}, the operator
$\mathcal{T}_{M}:\mathcal{B}\rightarrow\mathcal{B}$ given for any
$y\in\mathcal{\overline{D}}$ and $u\in\mathcal{B}$ by
\begin{equation}\label{optM}
\mathcal{T}_{M}(u)(y,\cdot)=
  \begin{cases}
   y\min\Big{\{}\frac{u(y,\cdot)}{y},M\Big{\}} & y\neq 0, \\
    0 & y=0.
  \end{cases}
  \end{equation}

We consider a truncated problem through the action of the
operator $\mathcal{T}_M$ on the gradient terms of the spde
\eqref{meqmain} for which we prove the next existence-uniqueness
theorem.
\begin{theorem}\label{t1}
Let the noise diffusion $\sigma$ satisfy
\begin{equation}\label{sigma1}
\sigma\in
C(\overline{\mathcal{D}}),\;\;\sigma(0)=\sigma(\ell)=0,\;\;\exists\;\sigma'(0).
\end{equation} Let also $M>0$ fixed, and some $p\geq p_0>8$,
and let $v_0(y)\in L^p(\Omega,C[0,T];\mathcal{B})$ be the initial
condition of \eqref{meqmain}. Then there exists a unique week
solution $(v^M,\eta^M)$ with $v^M\in
L^p(\Omega,C[0,T];\mathcal{B})$, depending on $M$, to the
truncated problem
\begin{equation}\label{Mrsde}
\begin{split}
&v^M_t(y,t)=\alpha\Delta v^M(y,t)-\nabla
(\mathcal{T}_M(v^M))(0^+,t)\nabla
(\mathcal{T}_M(v^M))(y,t)\\
&\hspace{1.7cm}+\sigma(y)\dot{W}(y,t)+\dot{\eta}^M(y,t),\;\;t\in(0,T],\;\;y\in\mathcal{D},\\
&v^M(y,0):=v_0(y),\;\;y\in\mathcal{D},\\
&v^M(0,t)=v^M(\lambda,t)=0,\;\;t\in(0,T],\\
\end{split}
\end{equation}
where $T:=T_M>0$ such that
\begin{equation}\label{stoparg}
\displaystyle{\sup_{r\in(0,T)}}|\nabla(\mathcal{T}_M(v^M))(0^+,r)|^p\leq
C_2(M,p)<\infty\;\;a.s.
\end{equation}
More specifically, for any $t\in(0,T)$, $v^M$ satisfies the week
formulation
\begin{equation}\label{nMg1*}
\begin{split}
v^M(y,t)=&\int_{\mathcal{D}}v_0(z)G(y,z,t)dz\\
&+\int_0^t\int_{\mathcal{D}}\nabla (\mathcal{T}_M(v^M))(0^+,s)\nabla G(y,z,t-s) \mathcal{T}_M(v^M)(z,s)dzds\\
&+\int_0^t\int_{\mathcal{D}} G(y,z,t-s)\sigma(z)W(dz,
ds)\\
&+\int_0^t\int_{\mathcal{D}} G(y,z,t-s)\eta^M(dz, ds),
\end{split}
\end{equation}
for $v^M(y,0):=v_0(y)$, and $\eta^M$, satisfies \eqref{psi} and
\eqref{refl}, i.e.,
\begin{equation}\label{npsi*}
\begin{split}
&\mbox{for all}\mbox{ measurable functions }
\psi:\;\overline{\mathcal{D}}\times(0,T)\rightarrow [0,\infty)\\
&\int_0^t\int_{\mathcal{D}}\psi(y,s)\eta^M(dy,ds)\;\;\mbox{is}\;\;\mathcal{F}_t-\mbox{measurable},
\end{split}
\end{equation}
and the constraint
\begin{equation}\label{nrefl*}
\int_0^T\int_{\mathcal{D}} v^M(y,s)\eta^M(dy,ds)=0.
\end{equation}
\end{theorem}
\begin{proof}
The operator $\mathcal{T}_M:\mathcal{B}\rightarrow\mathcal{B}$ is
well defined, \cite{BH}, and thus, for any $u$ in the space
$\mathcal{B}$, $\mathcal{T}_M(u)$ returns in $\mathcal{B}$, and so
$\mathcal{T}_M(u)\in C(\overline{\mathcal{D}})$ and vanishes at
the boundary of $\mathcal{D}$, while the gradient $\nabla
(\mathcal{T}_M(u))$ at $x=0$ exists.

Motivated by the integral representation \eqref{nMg1*}, we define
through a Picard iteration scheme the approximation $v_n^M$ of
$v^M$ as the solution of the approximate problem
\begin{equation}\label{nMg1}
\begin{split}
v_{n}^M(y,t)=&\int_{\mathcal{D}}v_0(z)G(y,z,t)dz\\
&+\int_0^t\int_{\mathcal{D}}\nabla (\mathcal{T}_M(v^M))(0^+,s)\nabla G(y,z,t-s) \mathcal{T}_M(v_{n-1}^M)(z,s)dzds\\
&+\int_0^t\int_{\mathcal{D}} G(y,z,t-s)\sigma(z)W(dz,
ds)\\
&+\int_0^t\int_{\mathcal{D}} G(y,z,t-s)\eta_n^M(dz,
ds),\;\;n:=1,2\cdots
\end{split}
\end{equation}
for $v_0^M(y,t):=v_0(y)$, 
 and $\eta_n^M$, which
approximates $\eta^M$, satisfying \eqref{psi} and \eqref{refl},
i.e.,
\begin{equation}\label{npsi}
\begin{split}
&\mbox{for all}\mbox{ measurable functions }
\psi:\;\overline{\mathcal{D}}\times(0,T)\rightarrow [0,\infty)\\
&\int_0^t\int_{\mathcal{D}}\psi(y,s)\eta_n^M(dy,ds)\;\;\mbox{is}\;\;\mathcal{F}_t-\mbox{measurable},
\end{split}
\end{equation}
and the constraint
\begin{equation}\label{nrefl}
\int_0^T\int_{\mathcal{D}} v_n^M(y,s)\eta_n^M(dy,ds)=0.
\end{equation}

In order to keep $v_{n}^M$ non-negative, and having in mind the
integral property \eqref{nrefl}, we will absorb the reflection
term $\eta_n^M$ in the Picard scheme \eqref{nMg1}, by spliting
$v_n^M$ as follows
\begin{equation}\label{vns}
v_n^M(y,t)=u_n(y,t)+\mathbb{O}_n(y,t),
\end{equation}
where $\mathbb{O}_n(y,t)$ solves in the weak sense the Heat
Equation Obstacle problem for any $y\in\mathcal{D}$, $t\in[0,T]$
\begin{equation}\label{ob}
\begin{split}
&\partial_t\mathbb{O}_n(y,t)=\alpha\Delta
\mathbb{O}_n(y,t)+\tilde{\eta}_n(dy,dt),\;\;u_n+\mathbb{O}_n\geq
0(\Leftrightarrow \mathbb{O}_n\geq -u_n),\\
&\mathbb{O}_n(0,t)=\mathbb{O}_n(\lambda,t)=0,\;\;\mathbb{O}_n(y,0)=0,\\
&\int_0^T\int_{\mathcal{D}}
(u_n(y,s)+\mathbb{O}_n(y,s))\tilde{\eta}_n(dy,ds)=0.
\end{split}
\end{equation}
Note that the above problem has a unique weak solution
$(\mathbb{O}_n,\tilde{\eta}_n)$ as long as $u_n$ exists and is
smooth. We observe that $\mathbb{O}_n(y,0)=0$ yields that
 $u_n(y,0)=v_n^M(y,0)=v_0(y)$.

We define $\eta_n^M:=\tilde{\eta}_n$, and as we shall see it
satisfies \eqref{nMg1} when $v_n^M$ satisfies \eqref{vns}.
Indeed, we replace $v_n^M=u_n(y,t)+\mathbb{O}_n(y,t)$ at the
left-hand side of \eqref{nMg1} and obtain for
$\eta_n^M:=\tilde{\eta}_n$
\begin{equation}\label{nMg2}
\begin{split}
u_{n}&(y,t)+\mathbb{O}_n(y,t)=\int_{\mathcal{D}}v_0(z)G(y,z,t)dz\\
&+\int_0^t\int_{\mathcal{D}}\nabla (\mathcal{T}_M(v^M))(0^+,s)\nabla G(y,z,t-s) \mathcal{T}_M(v_{n-1}^M)(z,s)dzds\\
&+\int_0^t\int_{\mathcal{D}} G(y,z,t-s)\sigma(z)W(dz,
ds)\\
&+\int_0^t\int_{\mathcal{D}} G(y,z,t-s)\eta_n^M(dz,
ds),\\
=&\int_{\mathcal{D}}v_0(z)G(y,z,t)dz\\
&+\int_0^t\int_{\mathcal{D}}\nabla (\mathcal{T}_M(v^M))(0^+,s)\nabla G(y,z,t-s) \mathcal{T}_M(v_{n-1}^M)(z,s)dzds\\
&+\int_0^t\int_{\mathcal{D}} G(y,z,t-s)\sigma(z)W(dz,
ds)\\
&+\int_0^t\int_{\mathcal{D}} G(y,z,t-s)\tilde{\eta}_n(dz,
ds)\;\;n:=1,2\cdots
\end{split}
\end{equation}
Since $\mathbb{O}_n$ solves in the weak sense \eqref{ob}, and
$\mathbb{O}_n(y,0)=0$, then using the same Green's function $G$
for the integral representation of $\mathbb{O}_n$, we see that the
last term of \eqref{nMg2} coincides with $\mathbb{O}_n(y,t)$, so
we obtain
\begin{equation}\label{nMg3}
\begin{split}
u_{n}(y,t)=&\int_{\mathcal{D}}v_0(z)G(y,z,t)dz\\
&+\int_0^t\int_{\mathcal{D}}\nabla (\mathcal{T}_M(v^M))(0^+,s)\nabla G(y,z,t-s) \mathcal{T}_M(v_{n-1}^M)(z,s)dzds\\
&+\int_0^t\int_{\mathcal{D}} G(y,z,t-s)\sigma(z)W(dz,
ds),\;\;n:=1,2\cdots
\end{split}
\end{equation}

We split now $v^M$ by
\begin{equation}\label{svm}
v^M(y,t)=u(y,t)+\mathbb{O}(y,t),
\end{equation}
 and set $\eta^M:=\tilde{\eta}$,
where $(\mathbb{O}(y,t),\tilde{\eta}(y,t))$ solves in the weak
sense the Heat Equation Obstacle problem for any
$y\in\mathcal{D}$, $t\in[0,T]$
\begin{equation}\label{ob0}
\begin{split}
&\partial_t\mathbb{O}(y,t)=\alpha\Delta
\mathbb{O}(y,t)+\tilde{\eta}(dy,dt),\;\;u+\mathbb{O}\geq
0(\Leftrightarrow \mathbb{O}\geq -u),\\
&\mathbb{O}(0,t)=\mathbb{O}(\lambda,t)=0,\;\;\mathbb{O}(y,0)=0,\\
&\int_0^T\int_{\mathcal{D}}
(u(y,s)+\mathbb{O}(y,s))\tilde{\eta}(dy,ds)=0.
\end{split}
\end{equation}
We observe that $\mathbb{O}(y,0)=0$ yields that
 $u(y,0)=v^M(y,0)=v_0(y)$, and as we argued for the derivation of
 \eqref{nMg3}, we obtain that $u$ satisfies
\begin{equation}\label{nMg3u}
\begin{split}
u(y,t)=&\int_{\mathcal{D}}v_0(z)G(y,z,t)dz\\
&+\int_0^t\int_{\mathcal{D}}\nabla (\mathcal{T}_M(v^M))(0^+,s)\nabla G(y,z,t-s) \mathcal{T}_M(v^M)(z,s)dzds\\
&+\int_0^t\int_{\mathcal{D}} G(y,z,t-s)\sigma(z)W(dz, ds).
\end{split}
\end{equation}

Using \eqref{nMg3} for $u_n$, $u_{n-1}$, by substraction, we get
for $n=2,3,\cdots$
\begin{equation}\label{nMg4}
\begin{split}
u_{n}&(y,t)-u_{n-1}(y,t)= \int_0^t\int_{\mathcal{D}}\Big{[}\nabla
(\mathcal{T}_M(v^M))(0^+,s)\mathcal{T}_M(v_{n-1}^M)(z,s)
\\
&-\nabla(\mathcal{T}_M(v^M))(0^+,s)\mathcal{T}_M(v_{n-2}^M)(z,s)\Big{]}\nabla
G(y,z,t-s)dzds\\
=&\int_0^t\int_{\mathcal{D}}\Big{[}\nabla
(\mathcal{T}_M(v^M))(0^+,s)z\min\Big{\{}\frac{v_{n-1}^M(z,s)}{z},M\Big{\}}
\\
&-\nabla(\mathcal{T}_M(v^M))(0^+,s)z\min\Big{\{}\frac{v_{n-2}^M(z,s)}{z},M\Big{\}}\Big{]}\nabla
G(y,z,t-s)dzds.
\end{split}
\end{equation}
In the above, we apply $\|\cdot\|_{\mathcal{B}}$-norm at both
sides and then take $p$-powers for some $p>0$, and then
$\displaystyle{\sup_{t\in(0,T)}}$, and then expectation, to obtain
for $n=2,3,\cdots$
\begin{equation}\label{nMg5}
\begin{split}
&E\Big{(}\displaystyle{\sup_{t\in(0,T)}}\|u_{n}(\cdot,t)-u_{n-1}(\cdot,t)\|^p_{\mathcal{B}}\Big{)}\\
=&
E\Big{(}\displaystyle{\sup_{t\in(0,T)}}\Big{\|}\int_0^t\int_{\mathcal{D}}\Big{[}\nabla
(\mathcal{T}_M(v^M))(0^+,s)z\min\Big{\{}\frac{v_{n-1}^M(z,s)}{z},M\Big{\}}
\\
&-\nabla(\mathcal{T}_M(v^M))(0^+,s)z\min\Big{\{}\frac{v_{n-2}^M(z,s)}{z},M\Big{\}}\Big{]}\nabla
G(y,z,t-s)dzds\Big{\|}^p_{\mathcal{B}}\Big{)}.
\end{split}
\end{equation}
In \cite{BH}, various useful bounds were proven in the norm
$\|\cdot\|_{\mathcal{B}}$ for the heat kernel $G$ defined
explicitly by a different series representation than the standard
 trigonometric series (bounds holding obviously true for $\alpha t$ in
place of the time variable $t$, and $\mathcal{D}=(0,\lambda)$ in
place of $(0,1)$ there). In particular, we use the estimate of
Proposition 4.4. therein, to derive directly for some constant
$c=c(T,p)>0$, that
\begin{equation*}
E(\displaystyle{\sup_{t\in(0,T)}}\|J\|^p_{\mathcal{B}})\leq
c(T,p)\int_0^TE(\displaystyle{\sup_{\tau\in(0,s)}}\|f(\cdot,\tau)\|^p_{\mathcal{B}})ds,
\end{equation*}
for
$$J(y,t):=\int_0^t\int_{\mathcal{D}}f(z,s)\nabla
G(y,z,t-s)dzds,$$ and
\begin{equation*}
\begin{split}
f(z,s):=&\nabla
(\mathcal{T}_M(v^M))(0^+,s)z\min\Big{\{}\frac{v_{n-1}^M(z,s)}{z},M\Big{\}}
\\
&-\nabla(\mathcal{T}_M(v^M))(0^+,s)z\min\Big{\{}\frac{v_{n-2}^M(z,s)}{z},M\Big{\}}.
\end{split}
\end{equation*}
Using the above in \eqref{nMg5} yields for $n=2,3,\cdots$
\begin{equation}\label{nMg6}
\begin{split}
&E\Big{(}\displaystyle{\sup_{t\in(0,T)}}\|u_{n}(\cdot,t)-u_{n-1}(\cdot,t)\|^p_{\mathcal{B}}\Big{)}\\
&\leq cC(T,p) \int_0^T
E\Big{(}\displaystyle{\sup_{\tau\in(0,s)}}\displaystyle{\sup_{z\in\mathcal{D}}}\Big{|}\nabla
(\mathcal{T}_M(v^M))(0^+,\tau)\min\Big{\{}\frac{v_{n-1}^M(z,\tau)}{z},M\Big{\}}
\\
&\hspace{4.5cm}-\nabla(\mathcal{T}_M(v^M))(0^+,\tau)\min\Big{\{}\frac{v_{n-2}^M(z,\tau)}{z},M\Big{\}}\Big{|}^p\Big{)}ds\\
=&cC(T,p) \int_0^T
E\Big{(}\displaystyle{\sup_{\tau\in(0,s)}}\displaystyle{\sup_{z\in\mathcal{D}}}\Big{|}\nabla
(\mathcal{T}_M(v^M))(0^+,\tau)\Big{|}^p\\
&\hspace{4.5cm}\Big{|}\min\Big{\{}\frac{v_{n-1}^M(z,\tau)}{z},M\Big{\}}
-\min\Big{\{}\frac{v_{n-2}^M(z,\tau)}{z},M\Big{\}}\Big{|}^p\Big{)}ds.
\end{split}
\end{equation}
Observing that $a\leq M$ and $b\leq M$ yields
$\min\{a,M\}-\min\{b,M\}=a-b$, while $a\geq M$ and $b\geq M$
yields $\min\{a,M\}-\min\{b,M\}=M-M=0\leq |a-b|$, while $a\leq M$
and $b\geq M$ yields $\min\{a,M\}-\min\{b,M\}=a-M\leq 0\leq
|a-b|$, we have
$$|\min\{a,M\}-\min\{b,M\}|\leq |a-b|,$$ and so, we obtain by
\eqref{nMg6} for $n=2,3,\cdots$
\begin{equation}\label{nMg7}
\begin{split}
&E\Big{(}\displaystyle{\sup_{t\in(0,T)}}\|u_{n}(\cdot,t)-u_{n-1}(\cdot,t)\|^p_{\mathcal{B}}\Big{)}\\
&\leq cC(T,p) \int_0^T
E\Big{(}\displaystyle{\sup_{\tau\in(0,s)}}\displaystyle{\sup_{z\in\mathcal{D}}}\Big{|}\nabla
(\mathcal{T}_M(v^M))(0^+,\tau)\Big{|}^p\\
&\hspace{4.5cm}\Big{|}\min\Big{\{}\frac{v_{n-1}^M(z,\tau)}{z},M\Big{\}}
-\min\Big{\{}\frac{v_{n-2}^M(z,\tau)}{z},M\Big{\}}\Big{|}^p\Big{)}ds\\
&\leq cC(T,p) \int_0^T
E\Big{(}\displaystyle{\sup_{\tau\in(0,s)}}|\nabla(\mathcal{T}_M(v^M))(0^+,\tau)|^p\displaystyle{\sup_{\tau\in(0,s)}}
\displaystyle{\sup_{z\in\mathcal{D}}}
\Big{|}\frac{v_{n-1}^M(z,\tau)}{z}-\frac{v_{n-2}^M(z,\tau)}{z}\Big{|}^p\Big{)}ds\\
&=cC(T,p) \int_0^T
E\Big{(}\displaystyle{\sup_{\tau\in(0,s)}}|(\nabla(\mathcal{T}_M(v^M))(0^+,\tau)|^p\displaystyle{\sup_{\tau\in(0,s)}}
\|v_{n-1}^M(\cdot,\tau)-v_{n-2}^M(\cdot,\tau)\|^p_{\mathcal{B}}\Big{)}ds.
\end{split}
\end{equation}
But by \eqref{vns}, it holds that
\begin{equation}\label{pls}
\begin{split}
\displaystyle{\sup_{\tau\in(0,s)}}
\displaystyle{\sup_{z\in\mathcal{D}}}
\Big{|}\frac{v_{n-1}^M(z,\tau)}{z}-\frac{v_{n-2}^M(z,\tau)}{z}\Big{|}^p\leq
&c(p) \displaystyle{\sup_{\tau\in(0,s)}}
\displaystyle{\sup_{z\in\mathcal{D}}}
\Big{|}\frac{u_{n-1}(z,\tau)}{z}-\frac{u_{n-2}(z,\tau)}{z}\Big{|}^p\\
&+c(p)\displaystyle{\sup_{\tau\in(0,s)}}
\displaystyle{\sup_{z\in\mathcal{D}}}
\Big{|}\frac{\mathbb{O}_{n-1}(z,\tau)}{z}-\frac{\mathbb{O}_{n-2}(z,\tau)}{z}\Big{|}^p\\
\leq &c(p) \displaystyle{\sup_{\tau\in(0,s)}}
\displaystyle{\sup_{z\in\mathcal{D}}}
\Big{|}\frac{u_{n-1}(z,\tau)}{z}-\frac{u_{n-2}(z,\tau)}{z}\Big{|}^p,
\end{split}
\end{equation}
where for the last inequality we used the stability bound in
$\displaystyle{\sup_{\tau\in(0,s)}}
\displaystyle{\sup_{z\in\mathcal{D}}}$ of the obstacle problem
solutions by the obstacle, cf. \cite{BH} in the proof of Theorem
3.2. So, for $s=T$, we obtain by \eqref{nMg7}, \eqref{pls}, and
for $n=2,3,\cdots$
\begin{equation}\label{nMg8}
\begin{split}
c(p)E\Big{(}&\displaystyle{\sup_{t\in(0,T)}}\|v_{n}^M(\cdot,t)-v_{n-1}^M(\cdot,t)\|^p_{\mathcal{B}}\Big{)}\leq
E\Big{(}\displaystyle{\sup_{t\in(0,T)}}\|u_{n}(\cdot,t)-u_{n-1}(\cdot,t)\|^p_{\mathcal{B}}\Big{)}\\
&\leq c(p)2^pC(T,p) \int_0^T
E\Big{(}\displaystyle{\sup_{\tau\in(0,s)}}|\nabla(\mathcal{T}_M(v^M))(0^+,\tau)|^p\displaystyle{\sup_{\tau\in(0,s)}}
\|v_{n-1}^M(\cdot,\tau)-v_{n-2}^M(\cdot,\tau)\|^p_{\mathcal{B}}\Big{)}ds\\
&\leq C(T,p) \int_0^T
E\Big{(}\displaystyle{\sup_{\tau\in(0,s)}}|\nabla(\mathcal{T}_M(v^M))(0^+,\tau)|^p\displaystyle{\sup_{\tau\in(0,s)}}
\|u_{n-1}(\cdot,\tau)-u_{n-2}(\cdot,\tau)\|^p_{\mathcal{B}}\Big{)}ds.
\end{split}
\end{equation}
We apply the same argumentation as for deriving the above
inequality, on \eqref{nMg3} now. By using that $v_0\in
L^p(\Omega,C[0,T];\mathcal{B})$, and the estimate of Proposition
4.3 from \cite{BH}, as $p>8>2$, we obtain for $n=1,2,\cdots$
\begin{equation}\label{nMg9}
\begin{split}
c(p)E\Big{(}&\displaystyle{\sup_{t\in(0,T)}}\|v_{n}^M(\cdot,t)\|^p_{\mathcal{B}}\Big{)}\leq
E\Big{(}\displaystyle{\sup_{t\in(0,T)}}\|u_{n}(\cdot,t)\|^p_{\mathcal{B}}\Big{)}\\
\leq& C(T,p) \int_0^T
E\Big{(}\displaystyle{\sup_{\tau\in(0,s)}}|\nabla(\mathcal{T}_M(v^M))(0^+,\tau)|^p\displaystyle{\sup_{\tau\in(0,s)}}
\|v_{n-1}^M(\cdot,\tau)\|^p_{\mathcal{B}}\Big{)}ds\\
&+CT\displaystyle{\sup_{\tau\in(0,T)}}\|\sigma(\cdot,\tau)\|^p_{\mathcal{B}}+C\\
\leq& C(T,p) \int_0^T
E\Big{(}\displaystyle{\sup_{\tau\in(0,s)}}|\nabla(\mathcal{T}_M(v^M))(0^+,\tau)|^p
\displaystyle{\sup_{\tau\in(0,s)}}
\|u_{n-1}(\cdot,\tau)\|^p_{\mathcal{B}}\Big{)}ds\\
&+CT\displaystyle{\sup_{\tau\in(0,T)}}\|\sigma(\cdot,\tau)\|^p_{\mathcal{B}}+C.
\end{split}
\end{equation}
Here, we used \eqref{vns} and the bound of the solution
$\mathbb{O}_n$ by the obstacle (comparing with the zero solution)
for the first
 and the third inequality. Moreover, since $p>8$ we applied the
estimate of Proposition 4.5 in \cite{BH} to bound the noise term.
In the above, note that since $\sigma$ satisfies \eqref{sigma1},
i.e., $\sigma\in C(\overline{\mathcal{D}})$,
$\sigma(0)=\sigma(\ell)=0$, and that it exists $\sigma'(0)$, then
$\|\sigma\|_{\mathcal{B}}$ is well defined. This assumption
models a zero volatility at the boundary of $\mathcal{D}$ in
accordance to the Dirichlet b.c. for the density of the
transactions $v$, i.e., the solution of \eqref{meqmain}, that
vanishes on $\partial\mathcal{D}$.

Thus, by \eqref{nMg9}, since $\sigma$ satisfies \eqref{sigma1},
$v_n^M,u_n$ stay in $L^p(\Omega,C[0,T];\mathcal{B})$ if
$v_0^M(y,t):=v_0(y)\in L^p(\Omega,C[0,T];\mathcal{B})$ and
$T=T_M$ such that \eqref{stoparg} holds true, i.e.,
\begin{equation*}
\displaystyle{\sup_{r\in(0,T)}}|\nabla(\mathcal{T}_M(v^M))(0^+,r)|^p\leq
C_2(M,p)<\infty\;\;a.s.
\end{equation*}

Furthermore, by \eqref{nMg8}, we get that $v_n^M$, $u_n$ are
Cauchy in $L^p(\Omega,C[0,T];\mathcal{B})$ while they also stay
in $L^p(\Omega,C[0,T];\mathcal{B})$, for $T=T_M$. So, by the
completeness of the Banach space $\mathcal{B}$ in this norm, both
they converge in $L^p(\Omega,C[0,T];\mathcal{B})$ to some unique
$\hat{v}^M,\hat{u}\in L^p(\Omega,C[0,T];\mathcal{B})$ as
$n\rightarrow\infty$. In details, by using \eqref{nMg8}, we obtain
\begin{equation}\label{nMg8new}
\begin{split}
E\Big{(}&\displaystyle{\sup_{t\in(0,T)}}\|v_{n}^M(\cdot,t)-v_{n-1}^M(\cdot,t)\|^p_{\mathcal{B}}\Big{)}\leq
CE\Big{(}\displaystyle{\sup_{t\in(0,T)}}\|u_{n}(\cdot,t)-u_{n-1}(\cdot,t)\|^p_{\mathcal{B}}\Big{)}\\
&\leq \tilde{c} \int_0^T
E\Big{(}\displaystyle{\sup_{\tau\in(0,s)}}|\nabla(\mathcal{T}_M(v^M))(0^+,\tau)
|^p\displaystyle{\sup_{\tau\in(0,s)}}
\|v_{n-1}^M(\cdot,\tau)-v_{n-2}^M(\cdot,\tau)\|^p_{\mathcal{B}}\Big{)}ds\\
&\leq \tilde{c} \int_0^T
E\Big{(}\displaystyle{\sup_{\tau\in(0,s)}}
\|v_{n-1}^M(\cdot,\tau)-v_{n-2}^M(\cdot,\tau)\|^p_{\mathcal{B}}\Big{)}ds\\
&\leq
c^{n-1}\int_0^T\int_0^{s_{n-1}}\int_0^{s_{n-2}}\cdots\int_0^{s_2}1d_{s_1}\cdots
d_{s_{n-2}}d_{s_{n-1}}E\Big{(}\displaystyle{\sup_{\tau\in(0,T)}}
\|v_{1}^M(\cdot,\tau)-v_{0}^M(\cdot,\tau)\|^p_{\mathcal{B}}\Big{)}\\
&\leq C\frac{c^{n-1}}{(n-1)!}\rightarrow
0\;\;\;as\;\;n\rightarrow\infty,
\end{split}
\end{equation}
where we used that $v_1^M,\;v_0^M\in
L^p(\Omega,C[0,T];\mathcal{B})$ since $v_n^M$ stays in
$L^p(\Omega,C[0,T];\mathcal{B})$ for all $n$ if
$v_0^M(y,t):=v_0(y)\in L^p(\Omega,C[0,T];\mathcal{B})$.
Therefore, $v_n^M$, $u_n$ are Cauchy in
$L^p(\Omega,C[0,T];\mathcal{B})$.

Moreover, we also obtain, as in \eqref{nMg9}
\begin{equation}\label{nMg10}
\begin{split}
E\Big{(}&\displaystyle{\sup_{t\in(0,T)}}\|v^M(\cdot,t)\|^p_{\mathcal{B}}\Big{)}\leq
CE\Big{(}\displaystyle{\sup_{t\in(0,T)}}\|u(\cdot,t)\|^p_{\mathcal{B}}\Big{)}\\
\leq &C \int_0^T
E\Big{(}\displaystyle{\sup_{\tau\in(0,s)}}|\nabla(\mathcal{T}_M(v^M))(0^+,\tau)|^p
\displaystyle{\sup_{\tau\in(0,s)}}
\|v^M(\cdot,\tau)\|^p_{\mathcal{B}}\Big{)}ds\\
&+CT\displaystyle{\sup_{\tau\in(0,T)}}\|\sigma(\cdot,\tau)\|^p_{\mathcal{B}}+C\\
\leq &C \int_0^T
E\Big{(}\displaystyle{\sup_{\tau\in(0,s)}}|\nabla(\mathcal{T}_M(v^M))(0^+,\tau)
|^p\displaystyle{\sup_{\tau\in(0,s)}}
\|u(\cdot,\tau)\|^p_{\mathcal{B}}\Big{)}ds\\
&+CT\displaystyle{\sup_{\tau\in(0,T)}}\|\sigma(\cdot,\tau)\|^p_{\mathcal{B}}+C.
\end{split}
\end{equation}
Therefore, by Gronwall's inequality on \eqref{nMg10}, and using
\eqref{stoparg}, \eqref{sigma1}, we arrive at
\begin{equation}\label{nMg11}
\begin{split}
E\Big{(}&\displaystyle{\sup_{t\in(0,T)}}\|v^M(\cdot,t)\|^p_{\mathcal{B}}\Big{)}
+E\Big{(}\displaystyle{\sup_{t\in(0,T)}}\|u(\cdot,t)\|^p_{\mathcal{B}}\Big{)}
\leq C.
\end{split}
\end{equation}
Thus, we get that $v^M,u\in L^p(\Omega,C[0,T];\mathcal{B})$.
By substraction of the integral representations \eqref{nMg3},
\eqref{nMg3u}, we obtain (as previously when deriving the 4th
inequality of \eqref{nMg8}), that $u_n,u$ satisfy for
$n=1,2,\cdots$
\begin{equation}\label{nMg88new}
\begin{split}
E\Big{(}&\displaystyle{\sup_{t\in(0,T)}}\|u_{n}(\cdot,t)-u(\cdot,t)\|^p_{\mathcal{B}}\Big{)}\\
&\leq C \int_0^T
E\Big{(}\displaystyle{\sup_{\tau\in(0,s)}}|\nabla(\mathcal{T}_M(v^M))(0^+,\tau)
|^p\displaystyle{\sup_{\tau\in(0,s)}}
\|u_{n-1}(\cdot,\tau)-u(\cdot,\tau)\|^p_{\mathcal{B}}\Big{)}ds\\
&\leq \frac{C^{n-1}}{(n-1)!}
E\Big{(}\displaystyle{\sup_{t\in(0,T)}}\|u_{1}(\cdot,t)-u(\cdot,t)\|^p_{\mathcal{B}}\Big{)}\rightarrow
0\;\;as\;\;n\rightarrow\infty.
\end{split}
\end{equation}
Here, we used once again \eqref{sigma1}, \eqref{stoparg}, the
argument for the last bound of \eqref{nMg8new}, together with the
fact that, as proven, $u_1,u\in L^p(\Omega,C[0,T];\mathcal{B})$.
Therefore, for $T=T_M$, $u_n\rightarrow u$ as
$n\rightarrow\infty$ in $L^p(\Omega,C[0,T];\mathcal{B})$. By
uniqueness of the limits, we get that $\hat{u}=u$ a.s.

Since $u$ exists uniquely, then the solution
$(\mathbb{O},\tilde{\eta})$ of \eqref{ob0} exists uniquely. But
it holds that
\begin{equation*}
\begin{split}
\displaystyle{\sup_{\tau\in(0,s)}}
\displaystyle{\sup_{z\in\mathcal{D}}}
\Big{|}\frac{v_{n}^M(z,\tau)}{z}-\frac{v^M(z,\tau)}{z}\Big{|}^p\leq
&c \displaystyle{\sup_{\tau\in(0,s)}}
\displaystyle{\sup_{z\in\mathcal{D}}}
\Big{|}\frac{u_{n}(z,\tau)}{z}-\frac{u(z,\tau)}{z}\Big{|}^p\\
&+c\displaystyle{\sup_{\tau\in(0,s)}}
\displaystyle{\sup_{z\in\mathcal{D}}}
\Big{|}\frac{\mathbb{O}_{n}(z,\tau)}{z}-\frac{\mathbb{O}(z,\tau)}{z}\Big{|}^p\\
\leq &c \displaystyle{\sup_{\tau\in(0,s)}}
\displaystyle{\sup_{z\in\mathcal{D}}}
\Big{|}\frac{u_{n}(z,\tau)}{z}-\frac{u(z,\tau)}{z}\Big{|}^p,
\end{split}
\end{equation*}
where again we used the stability bound in
$\displaystyle{\sup_{\tau\in(0,s)}}
\displaystyle{\sup_{z\in\mathcal{D}}}$ of the obstacle problem
solutions by the obstacle, see in \cite{BH} in the proof of
Theorem 3.2. So, for $T=T_M$, by taking expectation, since, as we
have shown, $u_n\rightarrow u$ in
$L^p(\Omega,C[0,T];\mathcal{B})$, we obtain that $v_n^M\rightarrow
v^M$ as $n\rightarrow\infty$ in $L^p(\Omega,C[0,T];\mathcal{B})$.
By uniqueness of the limits, we get that $\hat{v}^M=v^M$ a.s.

So, for all $T=T_M>0$ such that \eqref{stoparg} holds true, we
derive the following:
\begin{enumerate}
\item $v^M,u$ exist and belong in
$L^p(\Omega,C[0,T];\mathcal{B})$.
\item Since $u$ exists, we may define $\eta^M:=\tilde{\eta}$ a.s. for
$\tilde{\eta}$ the second term of the solution
$(\mathbb{O},\tilde{\eta})$ of \eqref{ob0}.
\item The pair $(v^M,\eta^M)$ exists, $v^M\in L^p(\Omega,C[0,T];\mathcal{B})$, and
is the week solution of the truncated problem \eqref{Mrsde} in
$(0,T)$, where $v^M$, $\eta^M$ satisfy the week formulation
\eqref{nMg1*} for $v^M(y,0):=v_0(y)$, where $\eta^M$ satisfies
\eqref{npsi*} and \eqref{nrefl*}.
\item The pair $(v^M,\eta^M)$ is unique. Indeed, uniqueness of the limit of $v_n^M$ showed that $v^M$ is unique. Uniqueness of
$\eta^M$ follows by the uniqueness of $\tilde{\eta}$ of the
obstacle problem \eqref{ob0}
 since as we have shown $u$ exists uniquely as the limit of $u_n$ in $L^p(\Omega,C[0,T];\mathcal{B})$.
\end{enumerate}
Therefore, there exists a unique solution $(v^M,\eta^N)$ of the
week formulation \eqref{nMg1*} with $v^M\in
L^p(\Omega,C[0,T];\mathcal{B})$ and with $\eta^M$ satisfying
\eqref{npsi*} and \eqref{nrefl*}, which completes the proof.
\end{proof}

\vspace{0.5cm}

We return to the $M$-independent problem \eqref{meqmain}, and we
shall prove that it admits a unique maximal solution by
concatenation of the solution of the $M$-truncated problem
\eqref{Mrsde}-\eqref{npsi*}-\eqref{nrefl*}. This is established by
the next Main Theorem.

\begin{theorem}\label{t2}
Let the noise diffusion $\sigma$ satisfy \eqref{sigma1}, and
$v_0(y)\in L^p(\Omega,C[0,T];\mathcal{B})$ for $p\geq p_0>8$.
Then, there exists a unique weak maximal solution $(v,\eta)$ to
the problem \eqref{meqmain}-\eqref{psi}-\eqref{refl} in the
maximal interval $[0,\displaystyle{\sup_{M>0}}\tau_M)$, where
\begin{equation}\label{rtauM}
\begin{split}
\tau_M:=&\inf\Big{\{}T\geq
0:\;\displaystyle{\sup_{r\in(0,T)}}|\nabla v(0^+,r)|\geq
M\Big{\}}\\
=&\inf\Big{\{}T\geq 0:\;\displaystyle{\sup_{r\in(0,T)}}\nabla
v(0^+,r)\geq M\Big{\}}.
\end{split}
\end{equation}
\end{theorem}
\begin{proof}

We note that for the reflected problem since $v\geq 0$ a.s. and
$v(0,t)=0$, then $\nabla v(0^+,t)\geq 0$ a.s. for any $t$. As we
mentioned, we continue to keep the absolute value on $\nabla
v(0^+,t)$ in this proof also in order to present a more general
result applicable to the 2d i.b.v. problem of \eqref{stefr1y}, and
to the problem without reflection of next section.

Let $v^M$ as in Theorem \ref{t1}. We observe first that by the
operator definition \eqref{optM}, and since $v^M(0,t)=0$, we have
\begin{equation}\label{rstoparg}
\begin{split}
\displaystyle{\sup_{r\in(0,T)}}&|\nabla(\mathcal{T}_M(v^M))(0^+,r)|^p\leq\min\Big{\{}\displaystyle{\sup_{r\in(0,T)}}|\nabla
v^M(0^+,r)|^p,M^p\Big{\}}\leq M^p,
\end{split}
\end{equation}
and so by \eqref{stoparg}, Theorem \ref{t1} holds also for any
$T=T(M)$ such that
\begin{equation}\label{2stoparg}
\min\Big{\{}\displaystyle{\sup_{r\in(0,T)}}|\nabla
v^M(0^+,r)|,M\Big{\}}\leq M<C_2(M,p)^{1/p}<\infty\;\;a.s.
\end{equation}

We fix $M>0$, and consider arbitrary $\tilde{M}>0$ such that
$\tilde{M}\leq M$. Thus, the weak solution $(v^M,\eta^M)$ of
\eqref{Mrsde}-\eqref{npsi*}-\eqref{nrefl*} solves weakly the
$M$-independent problem \eqref{meqmain}-\eqref{psi}-\eqref{refl}
(since they share the same initial condition), until the (random)
stopping time $\tau$ defined as follows
\begin{equation}\label{rtau}
\begin{split}
\tau:=&\inf\Big{\{}T\geq
0:\;\min\Big{\{}\displaystyle{\sup_{r\in(0,T)}}|\nabla
v^M(0^+,r)|,M\Big{\}}\geq \tilde{M}\Big{\}}\\
=&\inf\Big{\{}T\geq 0:\;\displaystyle{\sup_{r\in(0,T)}}|\nabla
v^M(0^+,r)|\geq \tilde{M}\Big{\}}.
\end{split}
\end{equation}

Let now an arbitrary deterministic $t>0$, then we have as in
\eqref{nMg8}
\begin{equation}\label{nMg8*}
\begin{split}
&cE\Big{(}\displaystyle{\sup_{s\in(0,\min\{t,\tau\})}}\|v^M(\cdot,s)-v^{\tilde{M}}(\cdot,s)\|^p_{\mathcal{B}}\Big{)}
\leq E\Big{(}\displaystyle{\sup_{s\in(0,\min\{t,\tau\})}}\|u(\cdot,s)-\tilde{u}(\cdot,s)\|^p_{\mathcal{B}}\Big{)}\\
&\leq cC(t,p) \int_0^{\min\{t,\tau\}}
E\Big{(}\displaystyle{\sup_{r\in(0,s)}}|\nabla(\mathcal{T}_{\tilde{M}}(v^{\tilde{M}}))(0^+,r)|^p\displaystyle{\sup_{r\in(0,s)}}
\|v^M(\cdot,r)-v^{\tilde{M}}(\cdot,r)\|^p_{\mathcal{B}}\Big{)}ds\\
&\leq cC(t,p) \int_0^{\min\{t,\tau\}}
E\Big{(}\displaystyle{\sup_{r\in(0,s)}}|\nabla(\mathcal{T}_{\tilde{M}}(v^{\tilde{M}}))(0^+,r)
|^p\displaystyle{\sup_{r\in(0,s)}}
\|u(\cdot,r)-\tilde{u}(\cdot,r)\|^p_{\mathcal{B}}\Big{)}ds\\
&\leq C(t,p)\tilde{M}^p \int_0^{\min\{t,\tau\}}
E\Big{(}\displaystyle{\sup_{r\in(0,s)}}
\|u(\cdot,r)-\tilde{u}(\cdot,r)\|^p_{\mathcal{B}}\Big{)}ds,
\end{split}
\end{equation}
for $\tilde{u}$ given by
\begin{equation}\label{nMg3u*}
\begin{split}
\tilde{u}(y,t)=&\int_{\mathcal{D}}v_0(z)G(y,z,t)dz\\
&+\int_0^t\int_{\mathcal{D}}\nabla (\mathcal{T}_{\tilde{M}}(v^{\tilde{M}}))(0^+,s)\nabla G(y,z,t-s) \mathcal{T}_{\tilde{M}}(v^{\tilde{M}})(z,s)dzds\\
&+\int_0^t\int_{\mathcal{D}} G(y,z,t-s)\sigma(z)W(dz, ds),
\end{split}
\end{equation}
while
\begin{equation*}
\begin{split}
&cC(t,p) \int_0^{\min\{t,\tau\}}
E\Big{(}\displaystyle{\sup_{r\in(0,s)}}|\nabla(\mathcal{T}_{\tilde{M}}(v^{\tilde{M}}))(0^+,r)|^p\displaystyle{\sup_{r\in(0,s)}}
\|v^M(\cdot,r)-v^{\tilde{M}}(\cdot,r)\|^p_{\mathcal{B}}\Big{)}ds\\
&\leq C(t,p)\tilde{M}^p \int_0^{\min\{t,\tau\}}
E\Big{(}\displaystyle{\sup_{r\in(0,s)}}
\|v^M(\cdot,r)-v^{\tilde{M}}(\cdot,r)\|^p_{\mathcal{B}}\Big{)}ds.
\end{split}
\end{equation*}
 Therefore, by Gronwall's inequality, we get
$$u(\cdot,s)=\tilde{u}(\cdot,s),\;\;\;\;v^M(\cdot,s)=v^{\tilde{M}}(\cdot,s)\;\;\;\forall\;s\leq\min\{t,\tau\},$$
and thus, by the uniqueness of the obstacle problem solution
$\eta^{\tilde{M}}$ when $\tilde{u}$ exists, we arrive at
$$v^M(\cdot,s)=v^{\tilde{M}}(\cdot,s),\;\;\;\eta^M(\cdot,s)=\eta^{\tilde{M}}(\cdot,s)
\;\;\;\forall\;s\leq\min\{t,\tau\}.$$ Since $t$ is a
deterministic arbitrary constant, the above yields that
$$v^M(\cdot,s)=v^{\tilde{M}}(\cdot,s),\;\;\;\eta^M(\cdot,s)=\eta^{\tilde{M}}(\cdot,s)\;\;\forall\;s\leq\tau\;\;a.s.$$
So, the weak solutions of the $M$-truncated problem
\eqref{Mrsde}-\eqref{npsi*}-\eqref{nrefl*} are consistent, and we
can proceed to concatenation.

Let us define the stochastic process $(v,\eta)$ such that for all
$M>0$ it coincides with the weak solution $(v^M,\eta^M)$ of the
$M$-truncated problem \eqref{Mrsde}-\eqref{npsi*}-\eqref{nrefl*}
until the stopping time
\begin{equation*}
\begin{split}
\tau_M=&\inf\Big{\{}T\geq
0:\;\displaystyle{\sup_{r\in(0,T)}}|\nabla v^M(0^+,r)|\geq M\Big{\}}\\
=&\inf\Big{\{}T\geq 0:\;\displaystyle{\sup_{r\in(0,T)}}|\nabla
v(0^+,r)|\geq M\Big{\}}.
\end{split}
\end{equation*}
By its definition, $(v(\cdot,s),\eta(\cdot,s))$ is a weak solution
of the $M$-independent problem
\eqref{meqmain}-\eqref{psi}-\eqref{refl}, for any
$s\in[0,\displaystyle{\sup_{M>0}}\tau_M)$, and $\tau_M$ is a
localising sequence. Then, $(v(\cdot,s),\eta(\cdot,s))$ is a
maximal weak solution of \eqref{meqmain}-\eqref{psi}-\eqref{refl},
since
$$\displaystyle{\lim_{t\rightarrow
\Big{(}\displaystyle{\sup_{M>0}}\tau_M\Big{)}^-}}\displaystyle{\sup_{r\in(0,t)}}|\nabla
v^M(0^+,r)|=\infty\;\;a.s.$$ Uniqueness of the maximal weak
solution $(v(\cdot,s),\eta(\cdot,s))$ for
$s\in[0,\displaystyle{\sup_{M>0}}\tau_M)$, follows from the
consistency of the solution of the $M$-truncated problem with
which by its definition coincides.
\end{proof}

Let us now consider Case 1. The above analysis is valid for both
i.b.v. problems of \eqref{stefr1y}, and due to Theorem \ref{t2},
and under its assumptions there exist unique weak maximal
solutions $(v_1,\eta_1)$, $(v_2,\eta_2)$ satisfying
\eqref{psi}-\eqref{refl} in the maximal interval
$[0,\displaystyle{\sup_{M>0}}\tau_{1M})$, where
\begin{equation}\label{rtauM1}
\begin{split}
\tau_{1M}:=&\inf\Big{\{}T\geq
0:\;\displaystyle{\sup_{r\in(0,T)}}(|\nabla v_1(0^+,r)|+|\nabla
v_2(0^+,r)|)\geq
M\Big{\}}\\
=&\inf\Big{\{}T\geq 0:\;\displaystyle{\sup_{r\in(0,T)}}(\nabla
v_1(0^+,r)+\nabla v_2(0^+,r))\geq M\Big{\}}.
\end{split}
\end{equation}
Recall that $\Omega=(a,b)$, $\lambda=b-a$, and $a\leq s^-(0)\leq
s^+(0)\leq b$. We need $a\leq s^-(t)\leq s^+(t)\leq b$ in order to
return to the initial variables. This will restrict the stopping
time. By using the Stefan condition \eqref{stefcond1} we obtain
$$\partial_ts^-(t)=-\nabla v_2(0^+,t)\leq 0,\;\;\partial_t
s^+(t)=-\nabla v_1(0^+,t)\leq 0,$$ and so
$$s^-(t)\leq s^-(0)\leq b,\;\;s^+(t)\leq s^+(0)\leq b,$$
so we need $a\leq s^-(t)\leq s^+(t)$ which yields
$$a\leq s^-(0)-\int_0^t\nabla v_2(0^+,s)ds\leq s^+(0)-\int_0^t\nabla
v_1(0^+,s)ds.$$ We define the stopping time
\begin{equation}\label{psp1}
\tau_{1s}:=\inf\Big{\{}T>0:\;\sup_{r\in(0,T)}|\nabla
v_1(0^+,r)-\nabla v_2(0^+,r)|\geq
(T)^{-1}(s^{+}(0)-s^{-}(0))\Big{\}},
\end{equation}
to keep the spread non-negative and
\begin{equation}\label{nha1}
\tau_{1}^*:=\inf\Big{\{}T>0:\;\sup_{r\in(0,T)}\nabla
v_2(0^+,r)\geq (T)^{-1}(s^{-}(0)-a)\Big{\}},
\end{equation}
to keep the spread area in $\mathcal{D}$.

So, the next theorem holds.
\begin{theorem}\label{t31}
Under the assumptions of Theorem \ref{t2}, and if the initial
spread satisfies $\lambda> s^+(0)-s^-(0)\geq 0$, then there exist
unique weak maximal solutions $(w_1,\eta_1)$, $(w_2,\eta_2)$ to
the reflected Stefan problem
\eqref{stefr}-\eqref{psi}-\eqref{refl}, and $w|_{x\geq s^+}=w_1$,
$w|_{x\leq s^-}=-w_2$, in the maximal interval
$\mathcal{I}_1:=[0,\hat{\tau})$ for $\hat{\tau}:=
\min\{\displaystyle{\sup_{M>0}}\tau_{1M},\tau_{1s},\tau_1^*\}$,
with $\tau_{1M},\tau_{1s},\tau_1^*$ given by \eqref{rtauM1},
\eqref{psp1}, \eqref{nha1} for which the spread $s^{+}(t)-s^-(t)$
defined by the Stefan condition \eqref{stefcond1} exists and stays
a.s. non-negative for any $t\in\mathcal{I}_1$.

\end{theorem}

We consider now Case 2. Due to Theorem \ref{t2}, and under its
assumptions there exist unique weak maximal solutions
$(v_1,\eta_1)$, $(v_2,\eta_2)$ satisfying
\eqref{psi}-\eqref{refl} in the maximal interval
$[0,\displaystyle{\sup_{M>0}}\tau_{1M})$ for $\tau_{1M}$ given by
\eqref{rtauM1}. We need $a\leq s^-(t)\leq s^+(t)\leq b$ in order
to return to the initial variables. By using the Stefan condition
\eqref{stefcond2} we obtain
$$\partial_ts^-(t)=\nabla v_2(0^+,t)\geq 0,\;\;\partial_t
s^+(t)=-\nabla v_1(0^+,t)\leq 0,$$ and so
$$a\leq s^-(0)\leq s^-(t),\;\;s^+(t)\leq s^+(0)\leq b,$$
so we need $s^-(t)\leq s^+(t)$ which yields
$$s^-(0)+\int_0^t\nabla v_2(0^+,s)ds\leq s^+(0)-\int_0^t\nabla
v_1(0^+,s)ds.$$ We define the stopping time
\begin{equation}\label{psp2}
\tau_{2s}:=\inf\Big{\{}T>0:\;\sup_{r\in(0,T)}(\nabla
v_1(0^+,r)+\nabla v_2(0^+,r))\geq
(T)^{-1}(s^{+}(0)-s^{-}(0))\Big{\}},
\end{equation}
to keep the spread non-negative, while the spread area stays in
$\mathcal{D}$ as the spread is decreasing.

So, the next theorem holds.
\begin{theorem}\label{t32}
Under the assumptions of Theorem \ref{t2}, and if the initial
spread satisfies $\lambda> s^+(0)-s^-(0)\geq 0$, then there exist
unique weak maximal solutions $(w_1,\eta_1)$, $(w_2,\eta_2)$ to
the reflected Stefan problem
\eqref{stefrr}-\eqref{psi}-\eqref{refl}, and $w|_{x\geq
s^+}=w_1$, $w|_{x\leq s^-}=w_2$, in the maximal interval
$\mathcal{I}_2:=[0,\hat{\tau})$ for $\hat{\tau}:=
\min\{\displaystyle{\sup_{M>0}}\tau_{1M},\tau_{2s}\}$, with
$\tau_{1M},\tau_{2s}$ given by \eqref{rtauM1}, \eqref{psp2},
 for which the spread $s^{+}(t)-s^-(t)$ defined by
the Stefan condition \eqref{stefcond2} exists and stays a.s.
non-negative for any $t\in\mathcal{I}_2$.

\end{theorem}

\section{The problem without reflection}
\subsection{Existence of maximal solutions}
We shall consider the unreflected initial and boundary value
problem for
\begin{equation}\label{meqmain3}
v_t(y,t)=\alpha \Delta v(y,t)-\nabla v(0^+,t)\nabla
v(y,t)+\sigma(y)\dot{W}(y,t),
\end{equation}
posed for any $y$ in $\mathcal{D}=(0,\lambda)$ for $t\in[0,T]$
with Dirichlet b.c., with $v(y,0)$ given.

In the proofs of the previous section we replace the reflection
measure by $0$ and keep as presented the absolute value on the
changing in general sign $\nabla v(0^+,t)$ (as $v$ may take
negative values), and we derive the next results.

\begin{theorem}\label{t1d}
Let the noise diffusion $\sigma$ satisfy the condition
\eqref{sigma1}, $M>0$ fixed, $p\geq p_0>8$,
and let $v_0(y)\in L^p(\Omega,C[0,T];\mathcal{B})$ be the initial
condition of \eqref{meqmain3}. Then there exists a unique week
solution $v^M\in L^p(\Omega,C[0,T];\mathcal{B})$ to the truncated
problem
\begin{equation}\label{Mrsded}
\begin{split}
&v^M_t(y,t)=\alpha\Delta v^M(y,t)-\nabla
(\mathcal{T}_M(v^M))(0^+,t)\nabla
(\mathcal{T}_M(v^M))(y,t)\\
&\hspace{2.5cm}+\sigma(y)\dot{W}(y,t),\;\;t\in(0,T],\;\;y\in\mathcal{D},\\
&v^M(y,0):=v_0(y),\;\;y\in\mathcal{D},\\
&v^M(0,t)=v^M(\lambda,t)=0,\;\;t\in(0,T],\\
\end{split}
\end{equation}
where $T:=T_M>0$ such that
\begin{equation}\label{stoparg}
\displaystyle{\sup_{r\in(0,T)}}|\nabla(\mathcal{T}_M(v^M))(0^+,r)|^p<\infty\;\;a.s.,
\end{equation}
where for any $t\in(0,T)$, $v^M$ satisfies the week formulation
\begin{equation}\label{nMg1*d}
\begin{split}
v^M(y,t)=&\int_{\mathcal{D}}v_0(z)G(y,z,t)dz\\
&+\int_0^t\int_{\mathcal{D}}\nabla (\mathcal{T}_M(v^M))(0^+,s)\nabla G(y,z,t-s) \mathcal{T}_M(v^M)(z,s)dzds\\
&+\int_0^t\int_{\mathcal{D}} G(y,z,t-s)\sigma(z)W(dz, ds),
\end{split}
\end{equation}
for $v^M(y,0):=v_0(y)$.
\end{theorem}

\begin{theorem}\label{t2d}
Let the noise diffusion $\sigma$ satisfy \eqref{sigma1}, and
$v_0(y)\in L^p(\Omega,C[0,T];\mathcal{B})$ for $p\geq p_0>8$.
Then, there exists a unique weak maximal solution $v$ to the
problem \eqref{meqmain3} in the maximal interval
$[0,\displaystyle{\sup_{M>0}}\tilde{\tau}_{M})$, where
\begin{equation}\label{rtauMd}\tilde{\tau}_{M}:=\inf\Big{\{}T\geq
0:\;\displaystyle{\sup_{r\in(0,T)}}|\nabla v(0^+,r)|\geq
M\Big{\}}.
\end{equation}
\end{theorem}

We consider now Case 3. Due to Theorem \ref{t2d}, and under its
assumptions there exist unique weak maximal solutions
$(v_1,\eta_1)$, $(v_2,\eta_2)$ satisfying
\eqref{psi}-\eqref{refl} in the maximal interval
$[0,\displaystyle{\sup_{M>0}}\tau_{3M})$ for $\tau_{3M}$ given by
\begin{equation}\label{rtauM3}
\begin{split}
\tau_{3M}:=&\inf\Big{\{}T\geq
0:\;\displaystyle{\sup_{r\in(0,T)}}(|\nabla v_1(0^+,r)|+|\nabla
v_2(0^+,r)|)\geq M\Big{\}}.
\end{split}
\end{equation}
We need $a\leq s^-(t)\leq s^+(t)\leq b$ in order to return to the
initial variables. By using the Stefan condition
\eqref{stefcond2} we obtain
$$\partial_ts^-(t)=\nabla v_2(0^+,t),\;\;\partial_t
s^+(t)=-\nabla v_1(0^+,t),$$ and we need
$$a\leq s^-(0)+\int_0^t\nabla v_2(0^+,s)ds\leq s^+(0)-\int_0^t\nabla
v_1(0^+,s)ds\leq b.$$ We define the stopping time
\begin{equation}\label{psp3}
\tau_{3s}:=\inf\Big{\{}T>0:\;\sup_{r\in(0,T)}|\nabla
v_1(0^+,r)+\nabla v_2(0^+,r))|\geq
(T)^{-1}(s^{+}(0)-s^{-}(0))\Big{\}},
\end{equation}
to keep the spread non-negative, and
\begin{equation}\label{nha3}
\tau_{3}^*:=\inf\Big{\{}T>0:\;\sup_{r\in(0,T)}|\nabla
v_2(0^+,r)|\geq
(T)^{-1}(s^{-}(0)-a)\Big{\}},
\end{equation}
\begin{equation}\label{nha33}
\tau_{3}^{**}:=\inf\Big{\{}T>0:\;\sup_{r\in(0,T)}|\nabla
v_1(0^+,r)|\geq (T)^{-1}(b-s^{+}(0))\Big{\}}\Big{\}},
\end{equation}
to keep the spread area in $\mathcal{D}$.

So, the next theorem holds.
\begin{theorem}\label{t33}
Under the assumptions of Theorem \ref{t2d}, and if the initial
spread satisfies $\lambda> s^+(0)-s^-(0)\geq 0$, then there exist
unique weak maximal solutions $w_1$, $w_2$ to the unreflected
Stefan problem \eqref{stefrrr}, and $w|_{x\geq s^+}=w_1$,
$w|_{x\leq s^-}=w_2$, in the maximal interval
$\mathcal{I}_3:=[0,\hat{\tau})$ for $\hat{\tau}:=
\min\{\displaystyle{\sup_{M>0}}\tau_{3M},\tau_{3s},\tau_3^*,\tau_3^{**}\}$,
with $\tau_{3M},\tau_{3s},\tau_3^*,\tau_3^{**}$ given by
\eqref{rtauM3}, \eqref{psp3}, \eqref{nha3}, \eqref{nha33}
 for which the spread $s^{+}(t)-s^-(t)$ defined by
the Stefan condition \eqref{stefcond2} exists and stays a.s.
non-negative for any $t\in\mathcal{I}_3$.

\end{theorem}

\subsection{Formal asymptotics} Let us consider the Stefan problem
\eqref{stef1} without reflection, and the multiplicative noise
 defined as $\sigma(y)\dot{W}(y,t)\equiv\dot{W}(t)$. We conjecture that
 the spread $s^+-s^-$ satisfies $\forall\;t>0$
 \begin{equation}\label{form1}
\begin{split}
&\partial_t(s^+-s^-)(t)=-\frac{2}{\lambda}w_{\infty},\\
&s^+(0)-s^-(0)=\mbox{given},
\end{split}
\end{equation}
 for
$w_\infty(t)$ the solution of the initial value stochastic
equation problem
\begin{equation}\label{form2}
\partial_t
w_\infty(t)=\frac{2\alpha}{\lambda^2}w_\infty(t)+\dot{W}(t),\;\;t>0,\;\;\;\;w_\infty(0)=\mbox{given}.
\end{equation}

We shall present some formal arguments for the derivation of the
above system. Let $w_\infty(t)$ be the mean field profile of the
diffusion when the domain is infinite. We define
$\Omega=\Omega_{\rm Liq}(t)\cup [s^-(t),s^+(t)]$ of diameter $\lambda>>0$. 
Moreover, we consider $\alpha>>0$ so that $\Delta w=0$ on
$\Omega_{\rm Liq}$ approximates the spde of the Stefan problem
\eqref{stef1}. Then $w$ such that
$$w(x,t)=\frac{w_\infty(t)}{\lambda}
(x-s^+(t))\;\;\;x\geq
s^+(t),\;\;\;w(x,t)=\frac{w_\infty(t)}{\lambda}
(-x+s^-(t))\;\;\;x\leq s^-(t)$$ satisfies exactly
\begin{equation}\label{sde16}
\begin{split}
&\Delta w(x,t)=0,\;\;t>0,\\
&w(s^\pm(t),t)=0,\\
&w(x,t)=w_\infty(t)\;\;\mbox{at}\;\;x=\lambda+s^+(t)\;\;\mbox{where}\;\; x-s^+(t)=\lambda,\\
&w(x,t)=w_\infty(t)\;\;\mbox{at}\;\;x=s^-(t)-\lambda\;\;\mbox{where}\;\; -x+s^-(t)=\lambda,\\
&\partial_t(s^{+}-s^{-})(t)=-(\nabla w)^+(s^+(t),t)+(\nabla
w)^-(s^-(t),t)=-\frac{2}{\lambda}w_\infty(t)\;\;t>0.
\end{split}
\end{equation}
Here, the value of $w=w_\infty$ at $x=s^\mp\mp\lambda$, for
$\lambda>>0$, approximates the condition $w\rightarrow w_\infty$
at infinite distance from the solid phase of
\cite{niet1,niet2,AKY}, i.e., convergence to the mean-field
solution.

From the above approximate problem \eqref{sde16} we keep the
spread evolution through $w_\infty$, i.e.,
$\partial_t(s^{+}-s^{-})(t)=-\frac{2}{\lambda}w_\infty(t)$, and
proceed by matching it to the asymptotics of the solution $w$ of
the stochastic parabolic equation \eqref{stef1} on $\Omega_{\rm
Liq}$. This will yield the sde for $w_\infty(t)$ that is missing.
We observe that $\lambda=|\Omega|\simeq |\Omega_{\rm Liq}|$ and so
$$w_\infty(t)\simeq \frac{1}{|\Omega|}\int_{\Omega_{\rm Liq}}
w(x,t)dx.$$ By differentiating in time we have
\begin{equation*}
\begin{split}
|\Omega|\partial_t w_\infty\simeq &\int_{\Omega_{\rm Liq}}
w_t(x,t)dx-\int_{\partial\Omega_{\rm
Liq}}\partial_t(s^{+}(t)-s^{-}(t))wdx=\alpha\int_{\Omega_{\rm
Liq}}
\Delta w(x,t)dx\\
&+\int_{\Omega_{\rm Liq}}\sigma\dot{W}_s(x,t)-
\int_{\partial\Omega_{\rm Liq}}\partial_t(s^{+}(t)-s^{-}(t))wdx
\simeq -\alpha((\nabla w)^+(s^+(t),t)\\
&-(\nabla w)^-(s^-(t),t))+\int_{\Omega_{\rm
Liq}}\sigma\dot{W}_s(x,t)
-\int_{\partial\Omega_{\rm Liq}}\partial_t(s^{+}(t)-s^{-}(t))wdx\\
=&-\alpha\partial_t(s^{+}(t)-s^{-}(t))+\int_{\Omega_{\rm
Liq}}\sigma\dot{W}_s(x,t) -\int_{\partial\Omega_{\rm
Liq}}\partial_t(s^{+}(t)-s^{-}(t))wdx,
\end{split}
\end{equation*}
where we assumed that as $x\rightarrow\pm\infty$ $w\rightarrow
w_\infty(t)$ and so $\grad w\rightarrow 0$ (this can be modeled
by a Neumann condition for $w$ at $\partial\Omega$ when the
diameter $\lambda$ of $\Omega$ is very large). Since $\alpha>>0$
the last term can be ignored, and therefore
\begin{equation*}
\begin{split}
|\Omega|\partial_t w_\infty\simeq
-\alpha\partial_t(s^{+}(t)-s^{-}(t))+\int_{\Omega_{\rm
Liq}}\sigma\dot{W}_s(x,t).
\end{split}
\end{equation*}
Replacing $\lambda=|\Omega|$, and
$\partial_t(s^{+}(t)-s^{-}(t))=-\frac{2}{\lambda}w_{\infty}(t)$,
we get for $\sigma=1$ and $\dot{W}_s(x,t)=\dot{W}(t)$,
$$
\partial_t w_\infty\simeq \frac{2\alpha}{\lambda^2}
w_\infty+\dot{W}(t),\;\;\;
\partial_t(s^+-s^-)(t)=-\frac{2}{\lambda}w_\infty(t),$$
with initial values $w_\infty(0)$ and $s^+(0)-s^-(0)$ (the initial
spread).

\begin{remark}
Setting $w_\infty:=\beta+W$ we derive the following equivalent
problem to \eqref{form1}, \eqref{form2}
\begin{equation*}
\begin{split}
&\partial_t(s^+-s^-)(t)=-\frac{2}{\lambda}(\beta(t)+W(t)),\\
&s^+(0)-s^-(0)=\mbox{given},
\end{split}
\end{equation*}
 for
$\beta(t)$ the solution of
\begin{equation*}
\partial_t
\beta(t)=\frac{2\alpha}{\lambda^2}(\beta(t)+W(t)),\;\;t>0,\;\;\;\;\beta(0)=\mbox{given}.
\end{equation*}

\end{remark}

\section*{Acknowledgment}
The research work was supported by the Hellenic Foundation for
Research and Innovation (H.F.R.I.) under the First Call for
H.F.R.I. Research Projects to support Faculty members and
Researchers. (Project Number: HFRI-FM17-45).

\bibliographystyle{plain}

\begin{thebibliography}{99}

{\color{black}

\bibitem{af}
 N.D. Alikakos, G. Fusco, Ostwald ripening for dilute
systems under quasistationary dynamics, Comm. Math. Phys., 238,
pp.429--479, 2003.



\bibitem{afk2}
N.D. Alikakos, G. Fusco and G. Karali, Ostwald ripening in two
dimensions- The rigorous derivation of the equations from
Mullins-Sekerka dynamics, J. Differential Equations, 205(1),
pp.1--49, 2004.

\bibitem{afk1}
 N.D. Alikakos, G. Fusco and G. Karali, The effect of the
geometry of the particle distribution in Ostwald Ripening, Comm.
Math. Phys., 238, pp.480--488, 2003.

}

\bibitem{AM}
Y. Amihud, H. Mendelson, The Effects of Beta, Bid-Ask Spread,
Residual Risk, and Size on Stock Returns,  The Journal of Finance,
44(2), pp.479--486, 1989.

\bibitem{ABiK}
D.C. Antonopoulou, M. Bitsaki, G.D. Karali, The multi-dimensional
Stochastic Stefan Financial Model for a portfolio of assets,
Discrete Contin. Dyn. Syst. B, \verb"doi:10.3934/dcdsb.2021118",
2021.



\bibitem{AKY}
D.C. Antonopoulou, G.D. Karali, N.K. Yip, On the parabolic Stefan
problem for Ostwald ripening with kinetic undercooling and
inhomogeneous driving force, J. Differential Equations, 252,
pp.4679--4718, 2012.

\bibitem{connectliq}
C. Balardy, An emprirical analysis of the bid-ask spread in the
German power continuous market, Chaire European Electricity
Market, Fondation Paris-Dauphine, Working paper \# 35, 2018.

\bibitem{Berd1}
M. Berd, R. Mashal, P. Wang, Defining, Estimating and Using
Credit Term Structures. Part 1: Consistent Valuation Measures,
\verb"arXiv:0912.4609", 2004.

\bibitem{Berd2}
M. Berd, R. Mashal, P. Wang, Defining, Estimating and Using
Credit Term Structures. Part 2: Consistent Valuation Measures,
\verb"arXiv:0912.4614", 2004.


\bibitem{Berd3}
M. Berd, R. Mashal, P. Wang, Defining, Estimating and Using
Credit Term Structures. Part 3: Consistent Valuation Measures,
\verb"arXiv:0912.4618", 2004.

\bibitem{Bleaney}
M. Bleaney, Z. Li, The performance of bid-ask spread estimators
under less than ideal conditions, Studies in Economics and
Finance, 32, pp.98--127, 2015.

\bibitem{mfkorott}
X. Chen, O. Linton, S. Schneeberger, Simple nonparametric
estimators for the bid-ask spread in the Roll model, The
Institute for Fiscal Studies Department of Economics, UCL cemmap
working paper CWP12/16, Economic and Social Research Council,
2016.

\bibitem{sc88fee}
K.J. Cohen, S.F. Maier, R.A. Schwartz, D.K. Whitcomb, Transaction
Costs, Order Placement Strategy, and Existence of the Bid-Ask
Spread, Journal of Political Economy, 89(2), pp.287--305, 1981.

\bibitem{CH1}
R. Courant, D. Hilbert, Methods of Mathematical Physics, Vol. 1,
Interscience Publishers, 1953.



\bibitem{Ek}
E. Ekstr\"om, Selected Problems in Financial Mathematics, PhD
Thesis, Uppsala Universitet, Sweden, 2004.

\bibitem{GP}
M. D. Gould, M. A. Porter, S. Williams, M. McDonald, D. J. Fenn
and S. D. Howison, Limit order books, Quant. Finance, 13,
 pp.1709--1742, 2013.


 \bibitem{gropbank}
R. Gropp, C.K. Sorensen, J-D Lichtenberger, The dynamics of bank
spreads and financial structure, European Central Bank, Working
Paper Series No. 714, 2007.

 \bibitem{Ekinci}
Z.C. Guloglu, C. Ekinci, A comparison of bid-ask spread proxies:
Evidence from Borsa Istanbul futures, Journal of Economics,
Finance and Accounting, 3(1), pp.244--254, 2016.

\bibitem{BH}
B. Hambly, J. Kalsi, Stefan problems for reflected SPDEs driven
by space-time white noise, \verb"arXiv:1806.04739v3", 2018.

\bibitem{BH2}
B. Hambly, J. Kalsi, A Reflected Moving Boundary Problem Driven by
Space-Time White Noise, \verb"arXiv:1805.10166v1", 2018.

\bibitem{Hasbrouckt}
J. Hasbrouck, Liquidity in the futures pits: Inferring market
dynamics from incomplete data, Journal of Financial and
Quantitative Analysis, 39(2), pp.305--326, 2004.

\bibitem{Hosaunders}
T. Ho, A. Saunders, The determinants of bank interest margins:
Theory and empirical evidence, Journal of Financial and
Quantitative Analysis, 16(4), pp.581--600, 1981.

\bibitem{evolutbidask}
T. Lim, V.L. Vath, J-M Sahut, S. Scotti, Bid-ask spread
modelling, a perturbation approach, hal-00574184v2, 2011.

\bibitem{anrewload8f}
A.W. Lo, A. Craig MacKinlay, J. Zhang, Econometric models of
limit-order executions, Journal of Financial Economics, 65,
pp.31--71, 2002.

\bibitem{RobertMcDonh786ald}
R. L. McDonald, Derivatives Market, Third Edition, Pearson, 2013.

\bibitem{trnacost}
M. Mnif, H. Pham, A model of optimal portfolio selection under
liquidity risk and price impact, Finance and Stochastics, 11,
pp.51--90 2007.

\bibitem{mu}
 M. M\"uller, Stochastic Stefan-type problem under
frst-order boundary conditions, Ann. Appl. Probab., 28,
pp.2335--2369, 2018.

\bibitem{niet1}
B. Niethammer, Derivation of the LSW-theory for Ostwald ripening
by homogenization methods, Arch. Rational Mech. Anal., 147,
pp.119-178, 1999.

\bibitem{niet2}
B. Niethammer, Approximation of coarsening models by
homogenization of Stefan problem, PhD thesis, University of Bonn,
available as Preprint SFB 256, No. 453, 1996.

\bibitem{domocane}
D. O'Kane, S. Sen, Credit spreads explained, Journal of Credit
Risk, 1(2), pp.61--78, 2005.

\bibitem{ps}
C. Parlour and D. Seppi, Handbook of Financial Intermediation \&
Banking, North-Holland (imprint of Elsevier), Amsterdam, eds. A.
Boot and A. Thakor, 2008.

\bibitem{liqspread}
V. Plerou, P. Gopikrishnan, H.E. Stanley, Quantifying
fluctuations in market liquidity: Analysis of the bid-ask spread,
Physical Review E, 71, 046131, 2005.

\bibitem{Rollmeth}
R. Roll,   A simple implicit measure of the effective bid-ask
spread in an efficient market, The Journal of Finance, 39(4),
pp.1127--1139, 1984.


\bibitem{Walsh}
J.B. Walsh, An introduction to stochastic partial differential
equations, Lecture Notes in Mathematics, Springer-Verlag, 1986.

\bibitem{zz}
Z. Zheng, Stochastic Stefan problems: Existence, uniqueness, and
modeling of market limit orders, PhD Thesis, University of
Illinois at Urbana-Champaign, 2012.

\end{thebibliography}

\end{document}